\documentclass[12pt,fleqn]{article}

\usepackage{amsmath,amsthm,amssymb}
\usepackage{geometry} 
\usepackage[pdfpagemode=UseNone,pdfstartview=FitH]{hyperref}

\newcommand{\abs}[1]{\left|#1\right|}
\newcommand{\bdry}[1]{\partial #1}
\newcommand{\closure}[1]{\overline{#1}}
\newcommand{\comp}{\circ}
\newcommand{\dint}{\ds{\int}}
\newcommand{\dist}[2]{\text{dist}\, (#1,#2)}
\newcommand{\ds}[1]{\displaystyle #1}
\newcommand{\eps}{\varepsilon}
\newcommand{\Fucik}{Fu\v c\'\i k }
\newcommand{\half}{\frac{1}{2}}
\newcommand{\id}[1][]{id_{\, #1}}
\newcommand{\ip}[3][]{\left(#2,#3\right)_{#1}}
\newcommand{\norm}[2][]{\left\|#2\right\|_{#1}}
\renewcommand{\O}{\text{O}}
\renewcommand{\o}{\text{o}}
\newcommand{\PS}[1]{$(\text{PS})_{#1}$}
\newcommand{\pnorm}[2][]{\if #1'' \left|#2\right|_p \else \left|#2\right|_{#1} \fi}
\newcommand{\restr}[2]{\left.#1\right|_{#2}}
\newcommand{\seq}[1]{\left(#1\right)}
\newcommand{\set}[1]{\left\{#1\right\}}
\newcommand{\sip}[3][]{(#2,#3)_{#1}}
\newcommand{\snorm}[2][]{\|#2\|_{#1}}
\newcommand{\R}{\mathbb R}
\newcommand{\vol}[1]{\left|#1\right|}

\newenvironment{enumroman}{\begin{enumerate}

}{\end{enumerate}}

\newtheorem{lemma}{Lemma}[section]
\newtheorem{proposition}[lemma]{Proposition}
\newtheorem{theorem}[lemma]{Theorem}

\theoremstyle{definition}
\newtheorem{definition}[lemma]{Definition}

\numberwithin{equation}{section}

\title{\bf New linking theorems with applications to critical growth elliptic problems with jumping nonlinearities}
\author{\bf Kanishka Perera\\
Department of Mathematical Sciences\\
Florida Institute of Technology\\
Melbourne, FL 32901, USA\\
\em kperera@fit.edu\\
[\bigskipamount]
\bf Caterina Sportelli\\
Dipartimento di Matematica\\
Università degli Studi di Bari Aldo Moro\\
Via E. Orabona 4, 70125 Bari, Italy\\
\em caterina.sportelli@uniba.it}
\date{}

\begin{document}

\maketitle

\begin{abstract}
We study critical growth elliptic problems with jumping nonlinearities. Standard linking arguments based on decompositions of $H^1_0(\Omega)$ into eigenspaces of $- \Delta$ cannot be used to obtain nontrivial solutions to such problems. We show that the associated variational functional admits certain linking structures based on splittings of $H^1_0(\Omega)$ into nonlinear submanifolds. In order to capture these linking geometries, we prove several generalizations of the classical linking theorem of Rabinowitz that are not based on linear subspaces. We then use these new linking theorems to obtain nontrivial solutions of our problems. Our abstract results are of independent interest and can be used to obtain nontrivial solutions of other types of problems with jumping nonlinearities as well.
\end{abstract}

\section{Introduction}

The purpose of this paper is to study the existence of nontrivial solutions to critical growth elliptic problems with jumping nonlinearities such as
\begin{equation} \label{1}
\left\{\begin{aligned}
- \Delta u & = bu^+ - au^- + |u|^{2^\ast - 2}\, u && \text{in } \Omega\\[10pt]
u & = 0 && \text{on } \bdry{\Omega},
\end{aligned}\right.
\end{equation}
where $\Omega$ is a bounded domain in $\R^N,\, N \ge 4$, $2^\ast = 2N/(N - 2)$ is the critical Sobolev exponent, $a, b > 0$, and $u^\pm = \max \set{\pm u,0}$ are the positive and negative parts of $u$, respectively. When $a = b = \lambda$, this problem reduces to the well-known Br{\'e}zis-Nirenberg problem
\begin{equation} \label{2}
\left\{\begin{aligned}
- \Delta u & = \lambda u + |u|^{2^\ast - 2}\, u && \text{in } \Omega\\[10pt]
u & = 0 && \text{on } \bdry{\Omega}
\end{aligned}\right.
\end{equation}
(see \cite{MR709644}). It was shown in \cite{MR709644} that problem \eqref{2} has a positive solution when $0 < \lambda < \lambda_1$, where $\lambda_1 > 0$ is the first Dirichlet eigenvalue of $- \Delta$ in $\Omega$. Capozzi et al. \cite{MR831041} have extended this result by proving the existence of a nontrivial solution when $N = 4$ and $\lambda > \lambda_1$ is not an eigenvalue, or $N \ge 5$ and $\lambda \ge \lambda_1$. The proofs of their results are based on a decomposition of $H^1_0(\Omega)$ into eigenspaces of $- \Delta$. In contrast, the asymptotic problem
\begin{equation} \label{3}
\left\{\begin{aligned}
- \Delta u & = bu^+ - au^- && \text{in } \Omega\\[10pt]
u & = 0 && \text{on } \bdry{\Omega}
\end{aligned}\right.
\end{equation}
associated with problem \eqref{1} is nonlinear and therefore its solution set is not a linear subspace of $H^1_0(\Omega)$ when it has nontrivial solutions. Therefore the arguments in \cite{MR831041} cannot be used to obtain nontrivial solutions of problem \eqref{1}. We will show that the variational functional
\[
E(u) = \half \int_\Omega |\nabla u|^2\, dx - \half \int_\Omega \left[a\, (u^-)^2 + b\, (u^+)^2\right] dx - \frac{1}{2^\ast} \int_\Omega |u|^{2^\ast} dx, \quad u \in H^1_0(\Omega)
\]
associated with problem \eqref{1} admits, depending on the location of the point $(a,b)$ in the plane, certain linking structures based on a splitting of $H^1_0(\Omega)$ into nonlinear submanifolds. In order to capture these linking geometries, we will prove several generalizations of the classical linking theorem of Rabinowitz \cite{MR0488128} that are not based on linear subspaces. We will then use these new linking theorems to obtain nontrivial solutions of problem \eqref{1} when $(a,b)$ lies in certain regions of the plane.

To state our linking theorems, let $E$ be a $C^1$-functional on a Banach space $X$. Denote by $H$ the class of homeomorphisms $h$ of $X$ onto itself such that $h$ and $h^{-1}$ map bounded sets into bounded sets. Recall that a mapping $\varphi : Y \to Z$ between linear spaces is positive homogeneous if $\varphi(tu) = t \varphi(u)$ for all $u \in Y$ and $t \ge 0$. Let $X = N \oplus M,\, u = v + w$ be a direct sum decomposition with $N$ finite dimensional and $M$ closed and nontrivial. For $\rho > 0$, let $S_\rho = \set{u \in X : \norm{u} = \rho}$. We have the following theorems.

\begin{theorem} \label{Theorem 4}
Let $\theta \in C(M,N)$ be a positive homogeneous map. Assume that there exist $\rho > 0$ and $e \in X \setminus N$ such that
\begin{equation} \label{12}
\sup_{v \in N}\, E(v) < \inf_{u \in A}\, E(u), \qquad \sup_{u \in Q}\, E(u) < \infty,
\end{equation}
where $A = \set{\theta(w) + w : w \in M \cap S_\rho}$ and $Q = \set{v + se : v \in N,\, s \ge 0}$. Then
\begin{equation} \label{13}
\inf_{u \in A}\, E(u) \le c := \inf_{h \in \widetilde{H}}\, \sup_{u \in h(Q)}\, E(u) \le \sup_{u \in Q}\, E(u),
\end{equation}
where $\widetilde{H} = \set{h \in H : \restr{h}{N} = \id}$. Moreover, if $E$ satisfies the {\em \PS{c}} condition, then $c$ is a critical value of $E$.
\end{theorem}

\begin{theorem} \label{Theorem 5}
Let $\tau \in C(N,M)$ be a positive homogeneous map and let $B = \{v + \tau(v) : v \in N\}$. Assume that there exist $\rho > 0$ and $e \in X \setminus B$ with $-e \notin B$ such that
\begin{equation} \label{14}
\sup_{u \in B}\, E(u) < \inf_{w \in A}\, E(w), \qquad \sup_{u \in Q}\, E(u) < \infty,
\end{equation}
where $A = M \cap S_\rho$ and $Q = \set{u + se : u \in B,\, s \ge 0}$. Then
\begin{equation} \label{15}
\inf_{w \in A}\, E(w) \le c := \inf_{h \in \widetilde{H}}\, \sup_{u \in h(Q)}\, E(u) \le \sup_{u \in Q}\, E(u),
\end{equation}
where $\widetilde{H} = \set{h \in H : \restr{h}{B} = \id}$. Moreover, if $E$ satisfies the {\em \PS{c}} condition, then $c$ is a critical value of $E$.
\end{theorem}

We will prove Theorems \ref{Theorem 4} and \ref{Theorem 5}, as well as a generalization of Theorem \ref{Theorem 4} (see Theorem \ref{Theorem 3}), in Section \ref{Linking}.

To state our results for problem \eqref{1}, recall that the set $\Sigma(- \Delta)$ consisting of points $(a,b) \in \R^2$ for which problem \eqref{3} has a nontrivial solution is called the Dancer-\Fucik spectrum of $- \Delta$ in $\Omega$. Dancer \cite{MR499709, MR628124} and \Fucik \cite{MR0447688} recognized its significance for the solvability of the problem
\[
\left\{\begin{aligned}
- \Delta u & = bu^+ - au^- + f(u) && \text{in } \Omega\\[10pt]
u & = 0 && \text{on } \bdry{\Omega}
\end{aligned}\right.
\]
for a continuous function $f$ satisfying $f(t)/t \to 0$ as $|t| \to \infty$. Denoting by $\seq{\lambda_l}$ the sequence of Dirichlet eigenvalues of $- \Delta$ in $\Omega$, $\Sigma(- \Delta)$ contains the points $(\lambda_l,\lambda_l)$ since problem \eqref{3} reduces to the Dirichlet eigenvalue problem for $- \Delta$ when $a = b$. When $N = 1$, \Fucik showed in \cite{MR0447688} that $\Sigma(- d^2/dx^2)$ consists of a sequence of hyperbolic-like curves passing through the points $(\lambda_l,\lambda_l)$, with one or two curves going through each point. When $N \ge 2$, $\Sigma(- \Delta)$ consists locally of curves emanating from the points $(\lambda_l,\lambda_l)$ (see \cite{MR1011156, MR1181350, MR1269657, MR658734, MR871108, MR965532, MR1077275, MR1484910, MR640779, MR1322614}). In particular, Schechter showed in \cite{MR1322614} that in the square
\[
Q_l = (\lambda_{l-1},\lambda_{l+1}) \times (\lambda_{l-1},\lambda_{l+1}),
\]
$\Sigma(- \Delta)$ contains two strictly decreasing curves
\[
C_l : b = \nu_{l-1}(a), \qquad C^l : b = \mu_l(a),
\]
with
\[
\nu_{l-1}(a) \le \mu_l(a), \qquad \nu_{l-1}(\lambda_l) = \lambda_l = \mu_l(\lambda_l),
\]
such that the points in $Q_l$ that are either below the lower curve $C_l$ or above the upper curve $C^l$ are not in $\Sigma(- \Delta)$, while the points between them may or may not belong to $\Sigma(- \Delta)$ when they do not coincide. We will recall the construction of these curves in an abstract setting in Section \ref{DF spectrum}. We have the following results.

\begin{theorem} \label{Theorem 1}
Let $N \ge 4$. If $(a,b) \in Q_l$ and
\[
b < \nu_{l-1}(a)
\]
for some $l \ge 2$, then problem \eqref{1} has a nontrivial solution.
\end{theorem}

\begin{theorem} \label{Theorem 2}
Let $N \ge 5$. If $(a,b) \in Q_l$ and
\[
b \ge \mu_l(a)
\]
for some $l \ge 2$, then problem \eqref{1} has a nontrivial solution.
\end{theorem}

We will prove Theorems \ref{Theorem 1} and \ref{Theorem 2} in Section \ref{T1&2}. The proofs will be based on two abstract existence results that we will prove in Section \ref{Abs} using Theorems \ref{Theorem 4} and \ref{Theorem 5} (see Theorems \ref{Theorem 7} and \ref{Theorem 8}). These abstract results are of independent interest and can be used to obtain nontrivial solutions of other problems with jumping nonlinearities as well. For example, consider the critical growth problem
\begin{equation} \label{64}
\left\{\begin{aligned}
- \Delta u & = bu^+ - au^- + (e^{u^2} - 1)\, u && \text{in } \Omega\\[10pt]
u & = 0 && \text{on } \bdry{\Omega},
\end{aligned}\right.
\end{equation}
where $\Omega$ is a bounded domain in $\R^2$. Using Theorems \ref{Theorem 7} and \ref{Theorem 8}, we will prove the following results for this problem in Section \ref{T9&10}.

\begin{theorem} \label{Theorem 9}
Let $N = 2$. If $(a,b) \in Q_l$ and
\[
b < \nu_{l-1}(a)
\]
for some $l \ge 2$, then problem \eqref{64} has a nontrivial solution.
\end{theorem}

\begin{theorem} \label{Theorem 10}
Let $N = 2$. If $(a,b) \in Q_l$ and
\[
b > \mu_l(a)
\]
for some $l \ge 2$, then problem \eqref{64} has a nontrivial solution.
\end{theorem}

\section{Preliminaries on the Dancer-\Fucik spectrum} \label{DF spectrum}

In this preliminary section we briefly recall the construction of the minimal and maximal curves of the Dancer-\Fucik spectrum in an abstract setting introduced in Perera and Schechter \cite[Chapter 4]{MR3012848}. Let $H$ be a Hilbert space with the inner product $\ip{\cdot}{\cdot}$ and the associated norm $\norm{\cdot}$. Recall that an operator $\varphi : H \to H$ is monotone if $\ip{\varphi(u) - \varphi(v)}{u - v} \ge 0$ for all $u, v \in H$ and that $\varphi \in C(H,H)$ is a potential operator if $\varphi = \Phi'$ for some functional $\Phi \in C^1(H,\R)$, called a potential for $\varphi$. Assume that there are positive homogeneous monotone potential operators $p, n \in C(H,H)$ such that
\[
p(u) + n(u) = u, \quad \ip{p(u)}{n(u)} = 0 \quad \forall u \in H.
\]
We use the suggestive notation $u^+ = p(u),\, u^- = - n(u)$, so that
\[
u = u^+ - u^-, \quad \sip{u^+}{u^-} = 0.
\]
This implies
\begin{equation} \label{23}
\norm{u}^2 = \snorm{u^+}^2 + \snorm{u^-}^2,
\end{equation}
in particular, $\snorm{u^\pm} \le \norm{u}$.

Let $A$ be a self-adjoint operator on $H$ with the spectrum $\sigma(A) \subset (0,\infty)$ and $A^{-1}$ compact. Then $\sigma(A)$ consists of isolated eigenvalues $\lambda_l,\, l \ge 1$ of finite multiplicities satisfying $0 < \lambda_1 < \cdots < \lambda_l < \cdots$. Moreover, $D = D(A^{1/2})$ is a Hilbert space with the inner product
\[
\ip[D]{u}{v} = \sip{A^{1/2}\, u}{A^{1/2}\, v} = \ip{Au}{v}
\]
and the associated norm
\[
\norm[D]{u} = \snorm{A^{1/2}\, u} = \ip{Au}{u}^{1/2}.
\]
We have
\[
\norm[D]{u}^2 = \ip{Au}{u} \ge \lambda_1 \ip{u}{u} = \lambda_1 \norm{u}^2 \quad \forall u \in H,
\]
so $D \hookrightarrow H$, and the embedding is compact since $A^{-1}$ is a compact operator. Let $E_l$ be the eigenspace of $\lambda_l$,
\[
N_l = \bigoplus_{j=1}^l E_j, \qquad M_l = N_l^\perp \cap D.
\]
Then $D = N_l \oplus M_l,\, u = v + w$ is an orthogonal decomposition with respect to both $\ip{\cdot}{\cdot}$ and $\ip[D]{\cdot}{\cdot}$. Moreover,
\begin{equation} \label{24}
\norm[D]{v}^2 = \ip{Av}{v} \le \lambda_l \ip{v}{v} = \lambda_l \norm{v}^2 \quad \forall v \in N_l
\end{equation}
and
\begin{equation} \label{25}
\norm[D]{w}^2 = \ip{Aw}{w} \ge \lambda_{l+1} \ip{w}{w} = \lambda_{l+1} \norm{w}^2 \quad \forall w \in M_l.
\end{equation}
We assume that $w^\pm \ne 0$ for all $w \in M_1 \setminus \set{0}$.

\newpage

The set $\Sigma(A)$ consisting of points $(a,b) \in \R^2$ for which the equation
\begin{equation} \label{2.1}
Au = bu^+ - au^-, \quad u \in D
\end{equation}
has a nontrivial solution is called the Dancer-\Fucik spectrum of $A$. It is a closed subset of $\R^2$ (see \cite[Proposition 4.4.3]{MR3012848}). Since equation \eqref{2.1} reduces to $Au = \lambda u$ when $a = b = \lambda$, $\Sigma(A)$ contains the points $(\lambda_l,\lambda_l)$.

It is easily seen that $A$ is a potential operator with the potential
\[
\half \ip{Au}{u} = \half \norm[D]{u}^2.
\]
The potentials of $p$ and $n$ are
\[
\half \ip{p(u)}{u} = \half\, \snorm{u^+}^2, \quad \half \ip{n(u)}{u} = \half\, \snorm{u^-}^2,
\]
respectively (see \cite[Proposition 4.3.2]{MR3012848}). So solutions of equation \eqref{2.1} coincide with critical points of the $C^1$-functional
\[
I(u,a,b) = \norm[D]{u}^2 - a\, \snorm{u^-}^2 - b\, \snorm{u^+}^2, \quad u \in D,
\]
and $(a,b) \in \Sigma(A)$ if and only if $I(\cdot,a,b)$ has a nontrivial critical point. Let
\[
Q_l = (\lambda_{l-1},\lambda_{l+1}) \times (\lambda_{l-1},\lambda_{l+1}), \quad l \ge 2.
\]
If $(a,b) \in Q_l$, then $I(v + y + w,a,b),\, v + y + w \in N_{l-1} \oplus E_l \oplus M_l$ is strictly concave in $v$ and strictly convex in $w$, i.e., for $v_1 \ne v_2 \in N_{l-1},\, w \in M_{l-1}$,
\[
I((1 - t)\, v_1 + t v_2 + w,a,b) > (1 - t)\, I(v_1 + w,a,b) + t I(v_2 + w,a,b) \quad \forall t \in (0,1),
\]
and for $v \in N_l,\, w_1 \ne w_2 \in M_l$,
\[
I(v + (1 - t)\, w_1 + t w_2,a,b) < (1 - t)\, I(v + w_1,a,b) + t I(v + w_2,a,b) \quad \forall t \in (0,1)
\]
(see \cite[Proposition 4.6.1]{MR3012848}).

\begin{proposition}[{\cite[Proposition 4.7.1, Corollary 4.7.3, \& Proposition 4.7.4]{MR3012848}}] \label{Proposition 3}
Let $(a,b) \in Q_l$.
\begin{enumroman}
\item There is a positive homogeneous map $\theta(\cdot,a,b) \in C(M_{l-1},N_{l-1})$ such that $v = \theta(w,a,b)$ is the unique solution of
    \[
    I(v + w,a,b) = \sup_{v' \in N_{l-1}} I(v' + w,a,b), \quad w \in M_{l-1}.
    \]
    Moreover, $\theta$ is continuous on $M_{l-1} \times Q_l$, $\theta(w,\lambda_l,\lambda_l) = 0$ for all $w \in M_{l-1}$, and $I'(v + w,a,b) \perp N_{l-1}$ if and only if $v = \theta(w,a,b)$.
\item There is a positive homogeneous map $\tau(\cdot,a,b) \in C(N_l,M_l)$ such that $w = \tau(v,a,b)$ is the unique solution of
    \[
    I(v + w,a,b) = \inf_{w' \in M_l}\, I(v + w',a,b), \quad v \in N_l.
    \]
    Moreover, $\tau$ is continuous on $N_l \times Q_l$, $\tau(v,\lambda_l,\lambda_l) = 0$ for all $v \in N_l$, and $I'(v + w,a,b) \perp M_l$ if and only if $w = \tau(v,a,b)$.
\end{enumroman}
\end{proposition}

Let $S = \set{u \in D : \norm[D]{u} = 1}$ be the unit sphere in $D$. Set
\begin{equation} \label{27}
n_{l-1}(a,b) = \inf_{w \in M_{l-1} \cap S}\, I(\theta(w,a,b) + w,a,b)
\end{equation}
and
\begin{equation} \label{28}
m_l(a,b) = \sup_{v \in N_l \cap S}\, I(v + \tau(v,a,b),a,b).
\end{equation}
Since $I(u,a,b)$ is nonincreasing in $a$ for fixed $u$ and $b$, and in $b$ for fixed $u$ and $a$, $n_{l-1}(a,b)$ and $m_l(a,b)$ are nonincreasing in $a$ for fixed $b$, and in $b$ for fixed $a$. Moreover, $n_{l-1}$ and $m_l$ are continuous on $Q_l$ and $n_{l-1}(\lambda_l,\lambda_l) = 0 = m_l(\lambda_l,\lambda_l)$ (see \cite[Lemma 4.7.6 \& Proposition 4.7.7]{MR3012848}). For $a \in (\lambda_{l-1},\lambda_{l+1})$, set
\begin{equation} \label{29}
\nu_{l-1}(a) = \sup \set{b \in (\lambda_{l-1},\lambda_{l+1}) : n_{l-1}(a,b) \ge 0}
\end{equation}
and
\begin{equation} \label{30}
\mu_l(a) = \inf \set{b \in (\lambda_{l-1},\lambda_{l+1}) : m_l(a,b) \le 0}.
\end{equation}

\begin{proposition}[{\cite[Theorem 4.7.9]{MR3012848}}]
Let $(a,b) \in Q_l$.
\begin{enumroman}
\item $\nu_{l-1}$ is a continuous and strictly decreasing function, $\nu_{l-1}(\lambda_l) = \lambda_l$, $(a,b) \in \Sigma(A)$ if $b = \nu_{l-1}(a)$, and $(a,b) \notin \Sigma(A)$ if $b < \nu_{l-1}(a)$.
\item $\mu_l$ is a continuous and strictly decreasing function, $\mu_l(\lambda_l) = \lambda_l$, $(a,b) \in \Sigma(A)$ if $b = \mu_l(a)$, and $(a,b) \notin \Sigma(A)$ if $b > \mu_l(a)$.
\item $\nu_{l-1}(a) \le \mu_l(a)$.
\end{enumroman}
\end{proposition}

Thus,
\[
C_l : b = \nu_{l-1}(a), \qquad C^l : b = \mu_l(a)
\]
are strictly decreasing curves in $Q_l$ that belong to $\Sigma(A)$. They both pass through the point $(\lambda_l,\lambda_l)$ and may coincide. The region $\set{(a,b) \in Q_l : b < \nu_{l-1}(a)}$ below the lower curve $C_l$ and the region $\set{(a,b) \in Q_l : b > \mu_l(a)}$ above the upper curve $C^l$ are free of $\Sigma(A)$. They are the minimal and maximal curves of $\Sigma(A)$ in $Q_l$ in this sense. Points in the region $\set{(a,b) \in Q_l : \nu_{l-1}(a) < b < \mu_l(a)}$ between $C_l$ and $C^l$, when it is nonempty, may or may not belong to $\Sigma(A)$.

\section{Linking} \label{Linking}

In this section we prove Theorems \ref{Theorem 4} and \ref{Theorem 5} and a generalization of Theorem \ref{Theorem 4} (see Theorem \ref{Theorem 3}). The following variant definition of linking was given in Schechter \cite{MR1636619} and is a refined version of a definition given in Schechter and Tintarev \cite{MR95k:58033}. Let $X$ be a Banach space. Denote by $\Phi$ the set of all mappings $\Gamma \in C(X \times [0,1],X)$ such that, writing $\Gamma(u,t) = \Gamma(t)\, u$, we have
\begin{enumroman}
\item $\Gamma(0) = \id$,
\item $\Gamma(t)$ is a homeomorphism of $X$ for all $t \in [0,1)$, and the mapping $\Gamma^{-1} : X \times [0,1) \to X,\, (u,t) \mapsto \Gamma(t)^{-1}\, u$ is continuous,
\item $\Gamma(1)\, X$ is a single point in $X$, and $\Gamma(t)\, u \to \Gamma(1)\, X$ as $t \to 1$, uniformly on bounded subsets of $X$,
\item for each bounded subset $A$ of $X$ and $t_0 \in [0,1)$,
    \[
    \sup_{(u,t) \in A \times [0,t_0]} \Big(\norm{\Gamma(t)\, u} + \norm{\Gamma^{-1}(t)\, u}\Big) < \infty.
    \]
\end{enumroman}

\begin{definition}[\cite{MR1636619}]
Let $A$ and $B$ be disjoint nonempty subsets of $X$. We say that $A$ links $B$ if
\[
\Gamma(A \times (0,1]) \cap B \ne \emptyset \quad \forall \Gamma \in \Phi.
\]
\end{definition}

Denote by $H$ the class of homeomorphisms $h$ of $X$ onto itself such that $h$ and $h^{-1}$ map bounded sets into bounded sets. The following proposition was proved in Schechter \cite{MR1636619}.

\begin{proposition}[{\cite[Proposition 2.5]{MR1636619}}] \label{Proposition 2}
If $A$ links $B$ and $h \in H$, then $h(A)$ links $h(B)$.
\end{proposition}

Let $E$ be a $C^1$-functional on $X$. Recall that a subset $B$ of $X$ is a cone if $tu \in B$ whenever $u \in B$ and $t \ge 0$. We will obtain Theorems \ref{Theorem 4}, \ref{Theorem 5}, and \ref{Theorem 3} as corollaries of the following more general result.

\begin{theorem} \label{Theorem 6}
Let $A$ and $B$ be disjoint nonempty subsets of $X$ such that $A$ links $B$. Assume that $B$ is a cone and there exists $e \in X \setminus B$ with $-e \notin B$ such that
\begin{equation} \label{16}
\sup_{u \in B}\, E(u) < \inf_{u \in A}\, E(u), \qquad \sup_{u \in Q}\, E(u) < \infty,
\end{equation}
where $Q = \set{u + se : u \in B,\, s \ge 0}$. Then
\begin{equation} \label{17}
\inf_{u \in A}\, E(u) \le c := \inf_{h \in \widetilde{H}}\, \sup_{u \in h(Q)}\, E(u) \le \sup_{u \in Q}\, E(u),
\end{equation}
where $\widetilde{H} = \set{h \in H : \restr{h}{B} = \id}$. Moreover, if $E$ satisfies the {\em \PS{c}} condition, then $c$ is a critical value of $E$.
\end{theorem}

\newpage

First we prove the following lemma.

\begin{lemma} \label{Lemma 2}
Let $B$ be a nonempty cone, let $e \in X \setminus B$ with $-e \notin B$, and let $Q = \{u + se : u \in B,\, s \ge 0\}$. Then no subset of $X \setminus Q$ can link $B$.
\end{lemma}

\begin{proof}
We will show that
\begin{equation} \label{18}
\Gamma((X \setminus Q) \times (0,1]) \cap B = \emptyset
\end{equation}
for the mapping $\Gamma \in \Phi$ given by
\[
\Gamma(u,t) = (1 - t)\, u - te, \quad (u,t) \in X \times [0,1].
\]
Suppose \eqref{18} does not hold. Then there exist $(u,t) \in (X \setminus Q) \times (0,1]$ and $v \in B$ such that
\[
(1 - t)\, u - te = v.
\]
Since $-e \notin B$, $t \ne 1$ and
\begin{equation} \label{19}
u = \frac{v}{1 - t} + \frac{te}{1 - t}.
\end{equation}
Since $v \in B$ and $B$ is a cone, $v/(1 - t) \in B$, so the right-hand side of \eqref{19} is in $Q$. This is a contradiction since $u \in X \setminus Q$.
\end{proof}

Next we prove the following intersection lemma.

\begin{lemma} \label{Lemma 3}
Let $A$ and $B$ be disjoint nonempty subsets of $X$ such that $A$ links $B$. Assume that $B$ is a cone, $e \in X \setminus B$ with $-e \notin B$, and let $Q = \set{u + se : u \in B,\, s \ge 0}$. Then
\[
h(Q) \cap A \ne \emptyset \quad \forall h \in \widetilde{H},
\]
where $\widetilde{H} = \set{h \in H : \restr{h}{B} = \id}$.
\end{lemma}

\begin{proof}
Suppose $h(Q) \cap A = \emptyset$ for some $h \in \widetilde{H}$. Then $h^{-1}(A) \subset X \setminus Q$. Since $A$ links $B$ and $h^{-1} \in H$, $h^{-1}(A)$ links $h^{-1}(B)$ by Schechter \cite[Proposition 2.5]{MR1636619}. Since $h \in \widetilde{H}$, $h^{-1} \in \widetilde{H}$ and hence $h^{-1}(B) = B$, so $h^{-1}(A)$ links $B$. This contradicts Lemma \ref{Lemma 2} since $h^{-1}(A) \subset X \setminus Q$.
\end{proof}

We are now ready to prove Theorem \ref{Theorem 6}.

\begin{proof}[Proof of Theorem \ref{Theorem 6}]
Since $h(Q) \cap A \ne \emptyset$ for all $h \in \widetilde{H}$ by Lemma \ref{Lemma 3} and the identity map is in $\widetilde{H}$, \eqref{17} holds. It follows from \eqref{16} and \eqref{17} that
\begin{equation} \label{9}
\sup_{u \in B}\, E(u) < c < \infty.
\end{equation}
Suppose $E$ satisfies the \PS{c} condition, but $c$ is a regular value of $E$. By the first deformation lemma, then there exist $\eps_0 > 0$ and, for each $\eps \in (0,\eps_0)$, an $\eta \in H$ such that
\begin{equation} \label{10}
E(u) < c - 2 \eps \implies \eta(u) = u
\end{equation}
and
\begin{equation} \label{11}
E(u) \le c + \eps \implies E(\eta(u)) \le c - \eps
\end{equation}
(see, e.g., Perera and Schechter \cite[Lemma 1.3.5]{MR3012848}). By \eqref{9} and \eqref{10}, taking $\eps$ sufficiently small, we may assume that $\restr{\eta}{B}$ is the identity and hence $\eta \in \widetilde{H}$. By the definition of $c$, there exists $h \in \widetilde{H}$ such that
\[
\sup_{u \in h(Q)}\, E(u) \le c + \eps.
\]
Then $\widetilde{h} := \eta \comp h \in \widetilde{H}$ and
\[
\sup_{u \in \widetilde{h}(Q)}\, E(u) \le c - \eps
\]
by \eqref{11}, contradicting the definition of $c$.
\end{proof}

Let $X = N \oplus M,\, u = v + w$ be a direct sum decomposition with $N$ finite dimensional and $M$ closed and nontrivial. For $\rho > 0$, let $S_\rho = \set{u \in X : \norm{u} = \rho}$. The following proposition was proved in Schechter \cite{MR1636619}.

\begin{proposition}[{\cite[Corollary 2.7]{MR1636619}}] \label{Proposition 1}
For any $\rho > 0$, $M \cap S_\rho$ links $N$.
\end{proposition}

We are now ready to prove Theorems \ref{Theorem 4} and \ref{Theorem 5}.

\begin{proof}[Proof of Theorem \ref{Theorem 4}]
It suffices to show that $A$ links $N$ by Theorem \ref{Theorem 6}. By Proposition \ref{Proposition 1}, $M \cap S_\rho$ links $N$. Since $\theta \in C(M,N)$, the mapping
\[
h : X \to X, \quad v + w \mapsto v + \theta(w) + w,
\]
with the inverse
\[
h^{-1} : X \to X, \quad v + w \mapsto v - \theta(w) + w,
\]
is in $H$. So $h(M \cap S_\rho) = A$ links $h(N) = N$ by Proposition \ref{Proposition 2}.
\end{proof}

\begin{proof}[Proof of Theorem \ref{Theorem 5}]
Since $\tau$ is positive homogeneous, $B$ is a cone, so it suffices to show that $A$ links $B$ by Theorem \ref{Theorem 6}. By Proposition \ref{Proposition 1}, $A$ links $N$. Since $\tau \in C(N,M)$, the mapping
\[
h : X \to X, \quad v + w \mapsto v + \tau(v) + w,
\]
with the inverse
\[
h^{-1} : X \to X, \quad v + w \mapsto v - \tau(v) + w,
\]
is in $H$. So $h(A) = A$ links $h(N) = B$ by Proposition \ref{Proposition 2}.
\end{proof}

\newpage

Finally we prove the following generalization of Theorem \ref{Theorem 4}.

\begin{theorem} \label{Theorem 3}
Let $\theta \in C(M,N)$ be a positive homogeneous map and let $T : N \to X$ be a bounded linear map. If $\norm{I - T}$ is sufficiently small, where $I$ is the identity map on $N$, and there exist $\rho > 0$ and $e \in X \setminus T(N)$ such that
\begin{equation} \label{7}
\sup_{u \in T(N)}\, E(u) < \inf_{u \in A}\, E(u), \qquad \sup_{u \in Q}\, E(u) < \infty,
\end{equation}
where $A = \set{\theta(w) + w : w \in M \cap S_\rho}$ and $Q = \set{u + se : u \in T(N),\, s \ge 0}$, then
\begin{equation} \label{8}
\inf_{u \in A}\, E(u) \le c := \inf_{h \in \widetilde{H}}\, \sup_{u \in h(Q)}\, E(u) \le \sup_{u \in Q}\, E(u),
\end{equation}
where $\widetilde{H} = \{h \in H : \restr{h}{T(N)} = \id\}$. Moreover, if $E$ satisfies the {\em \PS{c}} condition, then $c$ is a critical value of $E$.
\end{theorem}

First we prove the following linking result.

\begin{lemma} \label{Lemma 4}
Let $T : N \to X$ be a bounded linear map. If $\norm{I - T}$ is sufficiently small, where $I$ is the identity map on $N$, then $M \cap S_\rho$ links $T(N)$ for any $\rho > 0$.
\end{lemma}

\begin{proof}
By Proposition \ref{Proposition 1}, it suffices to show that $X = T(N) \oplus M$. Let $P_N : X \to N,\, u \mapsto v$ and $P_M : X \to M,\, u \mapsto w$ be the projections onto $N$ and $M$, respectively. Since $(I - P_N\, T)\, v = P_N\, (I - T)\, v$ for all $v \in N$,
\[
\norm{I - P_N\, T} = \norm{P_N\, (I - T)} \le \norm{P_N} \norm{I - T} \le \norm{I - T}.
\]
So if $\norm{I - T}$ is sufficiently small, then the mapping $P_N\, T : N \to N$ is invertible and hence
\[
P_N(T(N)) = N.
\]
So given $u = v + w \in N \oplus M$, there exists $z \in T(N)$ such that $v = P_N z$. Since $z = P_N z + P_M z$, then
\begin{equation} \label{3.1}
u = P_N z + w = z + (w - P_M z) \in T(N) \oplus M.
\end{equation}
Denoting by $S$ the unit sphere in $X$, $N \cap S$ is compact since $N$ is finite dimensional, and $M \cap S$ is closed since $M$ is closed, so
\[
\dist{N \cap S}{M \cap S} > 0.
\]
So if $\norm{I - T}$ is sufficiently small, then $T(N) \cap M = \set{0}$ and hence the decomposition in \eqref{3.1} is unique. Hence $X = T(N) \oplus M$.
\end{proof}

\newpage

We are now ready to prove Theorem \ref{Theorem 3}.

\begin{proof}[Proof of Theorem \ref{Theorem 3}]
We will show that $A$ links $T(N)$ if $\norm{I - T}$ is sufficiently small. The desired conclusions will then follow from Theorem \ref{Theorem 6}. Let $\Gamma \in \Phi$. Define $\widetilde{\Gamma} \in \Phi$ by
\[
\widetilde{\Gamma}(u,t) = \begin{cases}
v + 2t \theta(w) + w & \text{if } t \in [0,1/2]\\[7.5pt]
\Gamma(2t - 1)\, (v + \theta(w) + w) & \text{if } t \in (1/2,1]
\end{cases}
\]
for $u = v + w \in N \oplus M$. If $\norm{I - T}$ is sufficiently small, then $M \cap S_\rho$ links $T(N)$ by Lemma \ref{Lemma 4}, so
\begin{equation} \label{3.2}
\widetilde{\Gamma}((M \cap S_\rho) \times [0,1]) \cap T(N) \ne \emptyset.
\end{equation}
We have
\begin{equation} \label{3.3}
\widetilde{\Gamma}((M \cap S_\rho) \times [0,1]) = C_\rho \cup \Gamma(A \times (0,1]),
\end{equation}
where
\[
C_\rho = \set{t \theta(w) + w : w \in M \cap S_\rho,\, t \in [0,1]}.
\]
We will show that if $\norm{I - T}$ is sufficiently small, then
\begin{equation} \label{4}
C_\rho \cap T(N) = \emptyset.
\end{equation}
This together with \eqref{3.2} and \eqref{3.3} will show that
\[
\Gamma(A \times (0,1]) \cap T(N) \ne \emptyset
\]
and complete the proof.

Let $S$ be the unit sphere in $X$ and let $\pi : X \setminus \set{0} \to S,\, u \mapsto u/\norm{u}$ be the radial projection onto $S$. Since $\theta$ is positive homogeneous and $T$ is linear, to show that \eqref{4} holds, it suffices to show that
\begin{equation} \label{5}
\pi(C_\rho) \cap T(N) = \emptyset.
\end{equation}
We will show that $\pi(C_\rho)$ is closed. Since $N$ is finite dimensional, $N \cap S$ is compact, so this will imply that
\[
\dist{\pi(C_\rho)}{N \cap S} > 0
\]
and hence \eqref{5} holds if $\norm{I - T}$ is sufficiently small.

It remains to show that $\pi(C_\rho)$ is closed. Let
\[
D = \set{t \theta(w) + w : w \in M,\, t \in [0,1]}.
\]
Since
\[
\pi(C_\rho) = D \cap S,
\]
it suffices to show that $D$ is closed. Let $\seq{w_j} \subset M,\, \seq{t_j} \subset [0,1]$ be sequences such that
\[
t_j \theta(w_j) + w_j \to u = v + w \in N \oplus M.
\]
Since $M$ and $N$ are closed subspaces, the projection $P_M : X \to M,\, u \mapsto w$ is continuous, so
\[
w_j = P_M(t_j \theta(w_j) + w_j) \to P_M(u) = w.
\]
Since $\theta$ is continuous, then $\theta(w_j) \to \theta(w)$. For a renamed subsequence, $t_j$ converges to some $t \in [0,1]$. Then
\[
t_j \theta(w_j) + w_j \to t \theta(w) + w,
\]
so $u = t \theta(w) + w \in D$ and hence $D$ is closed.
\end{proof}

\section{Abstract existence results} \label{Abs}

In this section we prove two abstract existence results based on Theorems \ref{Theorem 4} and \ref{Theorem 5} in the setting of Section \ref{DF spectrum}. We consider the operator equation
\begin{equation} \label{22}
Au = bu^+ - au^- + f(u), \quad u \in D,
\end{equation}
where $a, b > 0$ and $f \in C(D,H)$ is a potential operator. Let $F \in C^1(D,\R)$ be the potential of $f$ that satisfies $F(0) = 0$, i.e.,
\[
F(u) = \int_0^1 \ip{f(su)}{u} ds, \quad u \in D
\]
(see \cite[Proposition 4.3.2]{MR3012848}). Solutions of equation \eqref{22} coincide with critical points of the $C^1$-functional
\[
E(u) = \half \norm[D]{u}^2 - \frac{a}{2}\, \snorm{u^-}^2 - \frac{b}{2}\, \snorm{u^+}^2 - F(u) = \half\, I(u,a,b) - F(u), \quad u \in D.
\]
We assume that
\begin{enumerate}
\item[$(F_1)$] $F(u) = \o(\norm[D]{u}^2)$ as $\norm[D]{u} \to 0$,
\item[$(F_2)$] $F(u) \ge 0$ for all $u \in D$,
\item[$(F_3)$] there exists $c^\ast > 0$ such that for each $c \in (0,c^\ast)$, every \PS{c} sequence of $E$ has a subsequence that converges weakly to a nontrivial critical point of $E$.
\end{enumerate}

\begin{theorem} \label{Theorem 7}
Assume $(F_1)$--$(F_3)$. If $(a,b) \in Q_l$,
\[
b < \nu_{l-1}(a),
\]
and there exists $e \in X \setminus N_{l-1}$ such that
\begin{equation} \label{20}
\sup_{u \in Q}\, E(u) < c^\ast,
\end{equation}
where $Q = \set{v + se : v \in N_{l-1},\, s \ge 0}$, then equation \eqref{22} has a nontrivial solution.
\end{theorem}

\begin{proof}
Since $b < \nu_{l-1}(a)$ and $\nu_{l-1}$ is continuous,
\[
b/(1 - \delta) \le \nu_{l-1}(a/(1 - \delta))
\]
if $\delta \in (0,1 - \max \set{a,b}/\lambda_{l+1})$ is sufficiently small. Then
\[
n_{l-1}(a/(1 - \delta),b/(1 - \delta)) \ge 0
\]
(see \eqref{29}) and hence
\begin{equation} \label{26}
I(\theta(w,a/(1 - \delta),b/(1 - \delta)) + w,a/(1 - \delta),b/(1 - \delta)) \ge 0 \quad \forall w \in M_{l-1}
\end{equation}
(see \eqref{27}). For $\rho > 0$ and $w \in M_{l-1} \cap S_\rho$, set $u = \theta(w,a/(1 - \delta),b/(1 - \delta)) + w$. Then
\[
E(u) = \frac{\delta}{2} \norm[D]{u}^2 + \frac{1 - \delta}{2} \left(\norm[D]{u}^2 - \frac{a}{1 - \delta}\, \snorm{u^-}^2 - \frac{b}{1 - \delta}\, \snorm{u^+}^2\right) - F(u),
\]
and the quantity inside the parentheses is equal to $I(u,a/(1 - \delta),b/(1 - \delta))$ and hence nonnegative by \eqref{26}, so
\[
E(u) \ge \frac{\delta}{2} \norm[D]{u}^2 + \o(\norm[D]{u}^2) \text{ as } \norm[D]{u} \to 0
\]
in view of $(F_1)$. Since
\[
\norm[D]{u}^2 = \norm[D]{\theta(w,a/(1 - \delta),b/(1 - \delta))}^2 + \norm[D]{w}^2
\]
and $\norm[D]{\theta(w,a/(1 - \delta),b/(1 - \delta))} = \O(\norm[D]{w})$ by positive homogeneity (see \cite[Proposition 4.3.1]{MR3012848}), it follows that
\[
E(u) \ge \frac{\delta}{2}\, \rho^2 + \o(\rho^2) \text{ as } \rho \to 0.
\]
So if $\rho > 0$ is sufficiently small,
\begin{equation} \label{31}
\inf_{u \in A}\, E(u) > 0,
\end{equation}
where $A = \set{\theta(w,a/(1 - \delta),b/(1 - \delta)) + w : w \in M_{l-1} \cap S_\rho}$.

Now we apply Theorem \ref{Theorem 4} with $M = M_{l-1}$, $N = N_{l-1}$, and $\theta = \theta(\cdot,a/(1 - \delta),b/(1 - \delta))$. Since $(a,b) \in Q_l$, $a, b > \lambda_{l-1}$ and hence
\begin{equation} \label{51}
I(v,a,b) \le \norm[D]{v}^2 - \lambda_{l-1} \left(\snorm{v^+}^2 + \snorm{v^-}^2\right) = \norm[D]{v}^2 - \lambda_{l-1} \norm{v}^2 \le 0 \quad \forall v \in N_{l-1}
\end{equation}
by \eqref{23} and \eqref{24}. This together with $(F_2)$ and the fact that $E(0) = 0$ implies
\[
\sup_{v \in N_{l-1}}\, E(v) = 0.
\]
Together with \eqref{31}, this gives the first inequality in \eqref{12}. The second inequality also holds by \eqref{20}, so \eqref{13} together with \eqref{31} and \eqref{20} gives
\[
0 < c := \inf_{h \in \widetilde{H}}\, \sup_{u \in h(Q)}\, E(u) < c^\ast,
\]
where $\widetilde{H} = \{h \in H : \restr{h}{N_{l-1}} = \id\}$. If $E$ satisfies the \PS{c} condition, then Theorem \ref{Theorem 4} gives a critical point of $E$ at the level $c$, which is nontrivial since $c > 0$. If, on the other hand, $E$ does not satisfy the \PS{c} condition, then $E$ has a \PS{c} sequence without a convergent subsequence, which then has a subsequence that converges weakly to a nontrivial critical point of $E$ by $(F_3)$.
\end{proof}

\begin{theorem} \label{Theorem 8}
Assume $(F_1)$--$(F_3)$ and let $B = \set{v + \tau(v,a,b) : v \in N_l}$. If $(a,b) \in Q_l$,
\[
b \ge \mu_l(a),
\]
and there exists $e \in X \setminus B$ with $-e \notin B$ such that
\begin{equation} \label{21}
\sup_{u \in Q}\, E(u) < c^\ast,
\end{equation}
where $Q = \set{u + se : u \in B,\, s \ge 0}$, then equation \eqref{22} has a nontrivial solution.
\end{theorem}

\begin{proof}
We apply Theorem \ref{Theorem 5} with $N = N_l$, $M = M_l$, and $\tau = \tau(\cdot,a,b)$. Since $b \ge \mu_l(a)$,
\[
m_l(a,b) \le 0
\]
(see \eqref{30}) and hence
\begin{equation} \label{52}
I(v + \tau(v,a,b),a,b) \le 0 \quad \forall v \in N_l
\end{equation}
(see \eqref{28}). This together with $(F_2)$ and the fact that $E(0) = 0$ implies
\begin{equation} \label{32}
\sup_{u \in B}\, E(u) = 0.
\end{equation}
Since $(a,b) \in Q_l$, $a, b \le \lambda_{l+1} - \delta$ if $\delta > 0$ is sufficiently small. Then
\[
I(w,a,b) \ge \norm[D]{w}^2 - (\lambda_{l+1} - \delta) \norm{w}^2 \ge \frac{\delta}{\lambda_{l+1}} \norm[D]{w}^2 \quad \forall w \in M_l
\]
by \eqref{23} and \eqref{25}, so
\[
E(w) \ge \frac{\delta}{2 \lambda_{l+1}} \norm[D]{w}^2 + \o(\norm[D]{w}^2) \text{ as } \norm[D]{w} \to 0
\]
in view of $(F_1)$. So if $\rho > 0$ is sufficiently small,
\begin{equation} \label{33}
\inf_{w \in A}\, E(w) > 0,
\end{equation}
where $A = M_l \cap S_\rho$. Together with \eqref{32}, this gives the first inequality in \eqref{14}. The second inequality also holds by \eqref{21}, so \eqref{15} together with \eqref{33} and \eqref{21} gives
\[
0 < c := \inf_{h \in \widetilde{H}}\, \sup_{u \in h(Q)}\, E(u) < c^\ast,
\]
where $\widetilde{H} = \set{h \in H : \restr{h}{B} = \id}$. If $E$ satisfies the \PS{c} condition, then Theorem \ref{Theorem 5} gives a critical point of $E$ at the level $c$, which is nontrivial since $c > 0$. If, on the other hand, $E$ does not satisfy the \PS{c} condition, then $E$ has a \PS{c} sequence without a convergent subsequence, which then has a subsequence that converges weakly to a nontrivial critical point of $E$ by $(F_3)$.
\end{proof}

\section{Proofs of Theorems \ref{Theorem 1} and \ref{Theorem 2}} \label{T1&2}

In this section we prove Theorems \ref{Theorem 1} and \ref{Theorem 2}. First we prove Theorem \ref{Theorem 1} for $N \ge 5$ and Theorem \ref{Theorem 2} using Theorems \ref{Theorem 7} and \ref{Theorem 8}, respectively. Then we prove Theorem \ref{Theorem 1} for $N = 4$ using Theorem \ref{Theorem 3}. Since $u^\pm = (-u)^\mp$, $u$ solves \eqref{3} (resp. \eqref{1}) if and only if $-u$ solves \eqref{3} (resp. \eqref{1}) with $a$ and $b$ interchanged. So $\Sigma(- \Delta)$ is symmetric about the line $a = b$ and we may assume without loss of generality that $a \le b$.

Problem \eqref{1} fits into the abstract setting of Sections \ref{DF spectrum} and \ref{Abs} with $H = L^2(\Omega)$, $p(u) = u^+$, $n(u) = - u^-$, $D = H^1_0(\Omega)$, and $A$ equal to the inverse of the solution operator $L^2(\Omega) \to H^1_0(\Omega),\, g \mapsto u = (- \Delta)^{-1} g$ of the problem
\[
\left\{\begin{aligned}
- \Delta u & = g(x) && \text{in } \Omega\\[10pt]
u & = 0 && \text{on } \bdry{\Omega}.
\end{aligned}\right.
\]
Since the embedding $H^1_0(\Omega) \hookrightarrow L^2(\Omega)$ is compact, $A^{-1}$ is compact on $L^2(\Omega)$. The potential
\[
F(u) = \frac{1}{2^\ast} \int_\Omega |u|^{2^\ast} dx, \quad u \in H^1_0(\Omega)
\]
clearly satisfies $(F_1)$ and $(F_2)$. It is also well-known that the associated variational functional
\[
E(u) = \half \int_\Omega |\nabla u|^2\, dx - \half \int_\Omega \left[a\, (u^-)^2 + b\, (u^+)^2\right] dx - \frac{1}{2^\ast} \int_\Omega |u|^{2^\ast} dx, \quad u \in H^1_0(\Omega)
\]
satisfies $(F_3)$ with
\[
c^\ast = \frac{1}{N}\, S_N^{N/2},
\]
where
\begin{equation} \label{34}
S_N = \inf_{u \in H^1_0(\Omega) \setminus \set{0}}\, \frac{\dint_\Omega |\nabla u|^2\, dx}{\left(\dint_\Omega |u|^{2^\ast} dx\right)^{2/2^\ast}}
\end{equation}
is the best Sobolev constant (see Gazzola and Ruf \cite[Lemma 1]{MR1441856}).

Recall that the infimum in \eqref{34} is attained on the functions
\[
u_\eps(x) = c_N \left(\frac{\eps}{\eps^2 + |x|^2}\right)^{(N-2)/2}, \quad \eps > 0,
\]
where $c_N = [N(N - 2)]^{(N-2)/4}$ (see Talenti \cite{MR0463908}). Fix $x_0 \in \Omega$ and $\mu_0 > 1/\dist{x_0}{\bdry{\Omega}}$. Let $\xi : [0,\infty) \to [0,1]$ be a smooth function such that $\xi(s) = 1$ for $s \le 1/4$ and $\xi(s) = 0$ for $s \ge 1/2$. Set
\[
u_{\eps,\,\mu}(x) = \xi(\mu\, |x - x_0|)\, u_\eps(x - x_0), \quad \eps > 0,\, \mu \ge \mu_0.
\]
We will apply Theorems \ref{Theorem 7}, \ref{Theorem 8}, and \ref{Theorem 3} taking $e = u_{\eps,\,\mu}$ with $\eps > 0$ sufficiently small and $\mu \ge \mu_0$ sufficiently large. We have the estimates
\begin{gather}
\label{37} \int_\Omega |\nabla u_{\eps,\,\mu}|^2\, dx \le S_N^{N/2} + c_1\, (\mu \eps)^{N-2},\\[15pt]
\label{45} \int_\Omega u_{\eps,\,\mu}^{2^\ast}\, dx \ge S_N^{N/2} - c_2\, (\mu \eps)^N,\\[15pt]
\label{38} \int_\Omega u_{\eps,\,\mu}^2\, dx \ge \begin{cases}
c_3\, \eps^2 - c_4\, \mu^{N-4}\, \eps^{N-2} & \text{if } N \ge 5\\[7.5pt]
c_3\, \eps^2 \abs{\log\, (\mu \eps)} - c_4\, \eps^2 & \text{if } N = 4,
\end{cases}\\[15pt]
\label{35} \int_\Omega u_{\eps,\,\mu}\, dx \le c_5\, \mu^{-2}\, \eps^{(N-2)/2},\\[15pt]
\label{36} \int_\Omega u_{\eps,\,\mu}^{2^\ast - 1}\, dx \le c_6\, \eps^{(N-2)/2}
\end{gather}
for some constants $c_1, \dots, c_6 > 0$ (estimates \eqref{37}--\eqref{38} can be found in Degiovanni and Lancelotti \cite{MR2514055}, and \eqref{35} and \eqref{36} are easily verified).

\subsection{Proofs of Theorem \ref{Theorem 1} for $\mathbf{N \ge 5}$ and Theorem \ref{Theorem 2}}

For $N \ge 5$, $(N + 2)/N < (N - 2)/2$. Fix
\begin{equation} \label{46}
\frac{N + 2}{N} < \beta < \frac{N - 2}{2}.
\end{equation}
Let $S = \set{u \in H^1_0(\Omega) : \norm{u} = 1}$.

\begin{lemma} \label{Lemma 5}
Let $K$ be a subset of $S \cap C^2(\Omega)$ such that
\begin{equation} \label{44}
\sup_{u \in K}\, \norm[C^2(\closure{B_{1/\mu_0}(x_0)})]{u} < \infty.
\end{equation}
Then there exist constants $c_7, \dots, c_{15} > 0$ such that for all $\eps > 0$, $\mu \ge \mu_0$, $u \in K$, and $s, t \ge 0$,
\begin{equation} \label{39}
\int_\Omega |\nabla (tu + su_{\eps,\,\mu})|^2\, dx \le \left(1 + c_7\, \mu^{-(N+2)}\right) t^2 + \left(S_N^{N/2} + c_8\, (\mu \eps)^{N-2}\right) s^2,
\end{equation}
\begin{multline} \label{40}
\int_\Omega |tu + su_{\eps,\,\mu}|^{2^\ast} dx \ge \left(\int_\Omega |u|^{2^\ast} dx - c_9\, \mu^{-N} - c_{10}\, \eps^{N\, [1 - 2 \beta/(N-2)]}\right) t^{2^\ast}\\[7.5pt]
+ \left(S_N^{N/2} - c_{11}\, (\mu \eps)^N - c_{12}\, \eps^{2N \beta/(N+2)}\right) s^{2^\ast},
\end{multline}
\begin{multline} \label{41}
\int_\Omega \left[a \left((tu + su_{\eps,\,\mu})^-\right)^2 + b \left((tu + su_{\eps,\,\mu})^+\right)^2\right] dx \ge \bigg(\int_\Omega \left[a\, (u^-)^2 + b\, (u^+)^2\right] dx\\[7.5pt]
- c_{13}\, \mu^{-4}\bigg)\, t^2 + \Big(c_{14}\, \eps^2 - c_{15}\, \mu^{N-4}\, \eps^{N-2}\Big)\, s^2.
\end{multline}
In particular,
\begin{multline*}
E(tu + su_{\eps,\,\mu}) \le \half\, \Big(I(u,a,b) + c_{16}\, \mu^{-4}\Big)\, t^2 - \frac{1}{2^\ast}\, \bigg(\int_\Omega |u|^{2^\ast} dx - c_9\, \mu^{-N}\\[7.5pt]
- c_{10}\, \eps^{N\, [1 - 2 \beta/(N-2)]}\bigg)\, t^{2^\ast} + \half \left(S_N^{N/2} - c_{14}\, \eps^2 + c_{17}\, (\mu \eps)^{N-2}\right) s^2 - \frac{1}{2^\ast}\, \Big(S_N^{N/2} - c_{11}\, (\mu \eps)^N\\[7.5pt]
- c_{12}\, \eps^{2N \beta/(N+2)}\Big)\, s^{2^\ast}
\end{multline*}
for some constants $c_{16}, c_{17} > 0$.
\end{lemma}

\begin{proof}
We have
\begin{equation} \label{42}
\int_\Omega |\nabla (tu + su_{\eps,\,\mu})|^2\, dx = t^2 + 2ts \int_\Omega \nabla u \cdot \nabla u_{\eps,\,\mu}\, dx + s^2 \int_\Omega |\nabla u_{\eps,\,\mu}|^2\, dx
\end{equation}
since $u \in S$. Since $u_{\eps,\,\mu} = 0$ on $\bdry{\Omega}$,
\begin{equation} \label{43}
\int_\Omega \nabla u \cdot \nabla u_{\eps,\,\mu}\, dx = - \int_\Omega u_{\eps,\,\mu}\, \Delta u\, dx \le c_7\, \mu^{-2}\, \eps^{(N-2)/2}
\end{equation}
for some constant $c_7 > 0$ by \eqref{44} and \eqref{35}. Combining \eqref{42} with \eqref{43} and \eqref{37}, and noting that
\[
2ts \mu^{-2}\, \eps^{(N-2)/2} = 2 \mu^{-(N+2)/2}\, t\, (\mu \eps)^{(N-2)/2}\, s \le \mu^{-(N+2)}\, t^2 + (\mu \eps)^{N-2}\, s^2,
\]
gives \eqref{39} for some constant $c_8 > 0$.

The elementary inequality
\[
|y + z|^p \ge |y|^p + |z|^p - p\, (p - 1)\, 2^{p-2} \left(|y|^{p-1}\, |z| + |y|\, |z|^{p-1}\right) \quad \forall y, z \in \R,\, p > 2
\]
together with \eqref{45}, \eqref{44}, \eqref{35}, and \eqref{36} gives
\begin{multline*}
\int_\Omega |tu + su_{\eps,\,\mu}|^{2^\ast} dx \ge t^{2^\ast} \int_\Omega |u|^{2^\ast} dx + s^{2^\ast} \int_\Omega u_{\eps,\,\mu}^{2^\ast}\, dx\\[7.5pt]
- 2^\ast (2^\ast - 1)\, 2^{2^\ast - 2}\, \bigg(t^{2^\ast - 1} s \int_\Omega |u|^{2^\ast - 1}\, u_{\eps,\,\mu}\, dx + ts^{2^\ast - 1} \int_\Omega |u|\, u_{\eps,\,\mu}^{2^\ast - 1}\, dx\bigg) \ge t^{2^\ast} \int_\Omega |u|^{2^\ast} dx\\[7.5pt]
+ s^{2^\ast} \left(S_N^{N/2} - c_2\, (\mu \eps)^N\right) - c_9\, \Big(t^{2^\ast - 1} s \mu^{-2}\, \eps^{(N-2)/2} + ts^{2^\ast - 1} \eps^{(N-2)/2}\Big)
\end{multline*}
for some constant $c_9 > 0$. Since
\[
t^{2^\ast - 1} s \mu^{-2}\, \eps^{(N-2)/2} = \mu^{-(N+2)/2}\, t^{2^\ast - 1}\, (\mu \eps)^{(N-2)/2}\, s \le \left(1 - \frac{1}{2^\ast}\right) \mu^{-N} t^{2^\ast} + \frac{1}{2^\ast}\, (\mu \eps)^N s^{2^\ast}
\]
and
\[
ts^{2^\ast - 1}\, \eps^{(N-2)/2} = \eps^{(N-2)/2 - \beta}\, t \eps^\beta\, s^{2^\ast - 1} \le \frac{1}{2^\ast}\, \eps^{N\, [1 - 2 \beta/(N-2)]}\, t^{2^\ast} + \left(1 - \frac{1}{2^\ast}\right) \eps^{2N \beta/(N+2)}\, s^{2^\ast}
\]
by Young's inequality, \eqref{40} follows.

Since $u_{\eps,\,\mu} = 0$ outside $B_{1/\mu}(x_0)$,
\begin{multline*}
\int_\Omega \left[a \left((tu + su_{\eps,\,\mu})^-\right)^2 + b \left((tu + su_{\eps,\,\mu})^+\right)^2\right] dx = t^2 \int_{\Omega \setminus B_{1/\mu}(x_0)} \left[a\, (u^-)^2 + b\, (u^+)^2\right] dx\\[7.5pt]
+ \int_{B_{1/\mu}(x_0)} \left[a \left((tu + su_{\eps,\,\mu})^-\right)^2 + b \left((tu + su_{\eps,\,\mu})^+\right)^2\right] dx = t^2 \int_\Omega \left[a\, (u^-)^2 + b\, (u^+)^2\right] dx\\[7.5pt]
+ \int_{B_{1/\mu}(x_0)} \left[a \left((tu + su_{\eps,\,\mu})^-\right)^2 + b \left((tu + su_{\eps,\,\mu})^+\right)^2 - a\, (tu^-)^2 - b\, (tu^+)^2\right] dx.
\end{multline*}
Since $a \le b$, the last integral is greater than or equal to
\begin{multline*}
\int_{B_{1/\mu}(x_0)} \left[a\, (tu + su_{\eps,\,\mu})^2 - b\, (tu)^2\right] dx = - (b - a)\, t^2 \int_{B_{1/\mu}(x_0)} u^2\, dx + 2ats \int_{B_{1/\mu}(x_0)}\! u u_{\eps,\,\mu}\, dx\\[7.5pt]
+ as^2 \int_{B_{1/\mu}(x_0)} u_{\eps,\,\mu}^2\, dx \ge - c_{13}\, t^2\, \mu^{-N} - 2c_{14}\, ts \mu^{-2}\, \eps^{(N-2)/2} + as^2\, \Big(c_3\, \eps^2 - c_4\, \mu^{N-4}\, \eps^{N-2}\Big)
\end{multline*}
for some constants $c_{13}, c_{14} > 0$ by \eqref{44}, \eqref{35}, and \eqref{38}. Since
\[
2ts \mu^{-2}\, \eps^{(N-2)/2} \le \mu^{-4}\, t^2 + \eps^{N-2}\, s^2
\]
and $N \ge 5$, \eqref{41} follows.
\end{proof}

Noting that $1/N < 1 - 2/(N - 2)$ for $N \ge 5$, now we fix
\begin{equation} \label{50}
\frac{1}{N} < \gamma < 1 - \frac{2}{N - 2}
\end{equation}
and take $\mu = \eps^{- \gamma}$. Then $(\mu \eps)^{N-2} = \eps^{(N-2)(1 - \gamma)}$ and $(N - 2)(1 - \gamma) > 2$ by \eqref{50}, so it follows from Lemma \ref{Lemma 5} that for all small $\eps > 0$, $u \in K$, and $s, t \ge 0$,
\begin{multline} \label{47}
E(tu + su_{\eps,\,\eps^{- \gamma}}) \le \half\, \Big(I(u,a,b) + c_{16}\, \eps^{4 \gamma}\Big)\, t^2 - \frac{1}{2^\ast}\, \bigg(\int_\Omega |u|^{2^\ast} dx - c_9\, \eps^{N \gamma}\\[7.5pt]
- c_{10}\, \eps^{N\, [1 - 2 \beta/(N-2)]}\bigg)\, t^{2^\ast} + \half \left(S_N^{N/2} - c_{18}\, \eps^2\right) s^2 - \frac{1}{2^\ast}\, \Big(S_N^{N/2} - c_{11}\, \eps^{N (1 - \gamma)}\\[7.5pt]
- c_{12}\, \eps^{2N \beta/(N+2)}\Big)\, s^{2^\ast}
\end{multline}
for some constant $c_{18} > 0$. The next lemma is crucial.

\begin{lemma} \label{Lemma 6}
Let $K$ be a subset of $S \cap C^2(\Omega)$ such that \eqref{44} holds and
\begin{equation} \label{48}
I(u,a,b) \le 0 \quad \forall u \in K.
\end{equation}
Then
\[
\sup_{u \in K,\, s, t \ge 0}\, E(tu + su_{\eps,\,\eps^{- \gamma}}) < \frac{1}{N}\, S_N^{N/2}
\]
for all sufficiently small $\eps > 0$.
\end{lemma}

\begin{proof}
For $u \in K$, \eqref{48} together with the assumption $a \le b$ and the H\"{o}lder inequality gives
\[
1 \le \int_\Omega \left[a\, (u^-)^2 + b\, (u^+)^2\right] dx \le b \int_\Omega u^2\, dx \le b \vol{\Omega}^{2/N} \left(\int_\Omega |u|^{2^\ast} dx\right)^{2/2^\ast},
\]
so
\begin{equation} \label{49}
\inf_{u \in K}\, \int_\Omega |u|^{2^\ast} dx > 0.
\end{equation}
It follows from \eqref{47}--\eqref{49} that for all sufficiently small $\eps > 0$, $u \in K$, and $s, t \ge 0$,
\begin{multline*}
E(tu + su_{\eps,\,\eps^{- \gamma}}) \le \left[\half \left(1 - c_{19}\, \eps^2\right) s^2 - \frac{1}{2^\ast}\, \big(1 - c_{20}\, \eps^{N (1 - \gamma)} - c_{21}\, \eps^{2N \beta/(N+2)}\big)\, s^{2^\ast}\right] S_N^{N/2}\\[7.5pt]
+ c_{22}\, \eps^{4 \gamma}\, t^2 - c_{23}\, t^{2^\ast}
\end{multline*}
for some constants $c_{19}, \dots, c_{23} > 0$. Maximizing the right-hand side over all $s, t \ge 0$ then gives
\[
E(tu + su_{\eps,\,\eps^{- \gamma}}) \le \frac{1}{N}\, \frac{\left(1 - c_{19}\, \eps^2\right)^{N/2}}{\left(1 - c_{20}\, \eps^{N (1 - \gamma)} - c_{21}\, \eps^{2N \beta/(N+2)}\right)^{(N-2)/2}}\, S_N^{N/2} + c_{24}\, \eps^{2N \gamma}
\]
for some constant $c_{24} > 0$. Since $2N \beta/(N + 2) > 2$ by \eqref{46}, and $N\, (1 - \gamma) > 2$ and $2N \gamma > 2 $ by \eqref{50}, the desired conclusion follows.
\end{proof}

We are now ready to prove Theorem \ref{Theorem 1} for $N \ge 5$ and Theorem \ref{Theorem 2}.

\begin{proof}[Proof of Theorem \ref{Theorem 1} for $N \ge 5$]
We apply Theorem \ref{Theorem 7} taking $e = u_{\eps,\,\eps^{- \gamma}}$ with $\eps > 0$ sufficiently small. Functions $u_{\eps,\,\eps^{- \gamma}}$ with different centers $x_0$ and sufficiently small $\eps$ have disjoint supports and are therefore linearly independent. Since $N_{l-1}$ is finite dimensional, it follows that $x_0 \in \Omega$ can be chosen so that $u_{\eps,\,\eps^{- \gamma}} \notin N_{l-1}$.

To verify \eqref{20}, we use Lemma \ref{Lemma 6} with $K = S \cap N_{l-1}$. As in the proof of Theorem \ref{Theorem 7}, \eqref{48} holds (see \eqref{51}). By the interior regularity of eigenfunctions, $N_{l-1} \subset C^2(\Omega)$. So the restrictions of functions in $N_{l-1}$ to $\closure{B_{1/\mu_0}(x_0)}$ form a finite dimensional subspace of $C^2(\closure{B_{1/\mu_0}(x_0)})$. Since the restrictions of functions in $K$ to $\closure{B_{1/\mu_0}(x_0)}$ is a subset of this subspace and is bounded in the $H^1$-norm, \eqref{44} follows.
\end{proof}

\begin{proof}[Proof of Theorem \ref{Theorem 2}]
We apply Theorem \ref{Theorem 8} taking $e = u_{\eps,\,\eps^{- \gamma}}$ with $\eps > 0$ sufficiently small. Gradients $I'(u_{\eps,\,\eps^{- \gamma}},a,b)$ with different centers $x_0$ and sufficiently small $\eps$ have disjoint supports and are therefore linearly independent. By Proposition \ref{Proposition 3}, $I'(v + \tau(v,a,b),a,b) \perp M_l$ for all $v \in N_l$ and hence $\set{I'(v + \tau(v,a,b),a,b) : v \in N_l}$ is a subset of $N_l$. Since $N_l$ is finite dimensional, it follows that $x_0 \in \Omega$ can be chosen so that $\pm u_{\eps,\,\eps^{- \gamma}} \notin B := \{v + \tau(v,a,b) : v \in N_l\}$.

To verify \eqref{21}, we use Lemma \ref{Lemma 6} with $K = S \cap B$. As in the proof of Theorem \ref{Theorem 8}, \eqref{48} holds (see \eqref{52}). It only remains to show that \eqref{44} holds. Let $u \in K$. By Proposition \ref{Proposition 3}, $I'(u,a,b) = z_u$ for some $z_u \in N_l$. Then
\begin{equation} \label{54}
\int_\Omega \nabla u \cdot \nabla \zeta\, dx - \int_\Omega \left(bu^+ - au^-\right) \zeta\, dx = \int_\Omega \nabla z_u \cdot \nabla \zeta\, dx \quad \forall \zeta \in H^1_0(\Omega),
\end{equation}
so $u$ is a weak solution of
\[
\left\{\begin{aligned}
- \Delta u & = bu^+ - au^- - \Delta z_u && \text{in } \Omega\\[10pt]
u & = 0 && \text{on } \bdry{\Omega}.
\end{aligned}\right.
\]
Testing \eqref{54} with $z_u$ gives
\[
\norm{z_u}^2 \le \norm{u} \norm{z_u} + \left(a\, |u^-|_2 + b\, |u^+|_2\right) \pnorm[2]{z_u},
\]
and since $K \subset S$, this implies that $\set{z_u : u \in K}$ is bounded in $H^1_0(\Omega)$. By the interior regularity of eigenfunctions, $N_l \subset C^{2,\alpha}(\Omega)$. So the restrictions of functions in $N_l$ to $\closure{B_{1/\mu_0}(x_0)}$ form a finite dimensional subspace of $C^{2,\alpha}(\closure{B_{1/\mu_0}(x_0)})$. Since the restrictions of functions in the set $\set{z_u : u \in K}$ to $\closure{B_{1/\mu_0}(x_0)}$ is a subset of this subspace and is bounded in the $H^1$-norm, it follows that this set is also bounded in $C^{2,\alpha}(\closure{B_{1/\mu_0}(x_0)})$. Now it follows from standard arguments in elliptic regularity theory that $K$ is bounded in $C^{2,\alpha}(\closure{B_{1/\mu_0}(x_0)})$.
\end{proof}

\subsection{Proof of Theorem \ref{Theorem 1} for $\mathbf{N = 4}$}

Let $\eta : [0,\infty) \to [0,1]$ be a smooth function such that $\eta(s) = 0$ for $s \le 3/4$, $\eta(s) = 1$ for $s \ge 1$, and $|\eta'(s)| \le 5$ for all $s$. Set
\[
v_\mu(x) = \eta(\mu\, |x - x_0|)\, v(x), \quad v \in N_{l-1},\, \mu \ge \mu_0.
\]
We will apply Theorem \ref{Theorem 3} taking $T$ to be the bounded linear map from $N_{l-1}$ to $H^1_0(\Omega)$ given by $Tv = v_\mu$ with $\mu \ge \mu_0$ sufficiently large. The critical exponent $2^\ast = 4$ now and the estimates \eqref{37}--\eqref{38} reduce to
\begin{gather}
\label{60} \int_\Omega |\nabla u_{\eps,\,\mu}|^2\, dx \le S_4^2 + c_1\, (\mu \eps)^2,\\[15pt]
\int_\Omega u_{\eps,\,\mu}^4\, dx \ge S_4^2 - c_2\, (\mu \eps)^4,\\[15pt]
\label{62} \int_\Omega u_{\eps,\,\mu}^2\, dx \ge c_3\, \eps^2 \abs{\log\, (\mu \eps)} - c_4\, \eps^2.
\end{gather}

\begin{lemma} \label{Lemma 7}
There exist constants $c_{25}, \dots, c_{28} > 0$ such that for all $\mu \ge \mu_0$ and $v \in N_{l-1} \cap S$,
\begin{gather}
\label{56} \int_\Omega |\nabla (v_\mu - v)|^2\, dx \le c_{25}\, \mu^{-2},\\[15pt]
\label{57} \int_\Omega |\nabla v_\mu|^2\, dx \le \int_\Omega |\nabla v|^2\, dx + c_{26}\, \mu^{-2},\\[15pt]
\int_\Omega v_\mu^4\, dx \ge \int_\Omega v^4\, dx - c_{27}\, \mu^{-4},\\[15pt]
\label{59} \int_\Omega \left[a\, (v_\mu^-)^2 + b\, (v_\mu^+)^2\right] dx \ge \int_\Omega \left[a\, (v^-)^2 + b\, (v^+)^2\right] dx - c_{28}\, \mu^{-4}.
\end{gather}
In particular, for all $t \ge 0$,
\[
E(tv_\mu) \le \half\, \Big(I(v,a,b) + c_{29}\, \mu^{-2}\Big)\, t^2 - \frac{1}{4} \left(\int_\Omega v^4\, dx - c_{27}\, \mu^{-4}\right) t^4
\]
for some constant $c_{29} > 0$.
\end{lemma}

\begin{proof}
We have
\[
\nabla v_\mu(x) = \eta(\mu\, |x - x_0|)\, \nabla v(x) + \mu\, \eta'(\mu\, |x - x_0|)\, \frac{x - x_0}{|x - x_0|}\, v(x)
\]
and hence $|\nabla (v_\mu - v)| \le |\nabla v| + 5 \mu\, |v|$. Since $v_\mu = v$ outside $B_{1/\mu}(x_0)$, this gives
\[
\int_\Omega |\nabla (v_\mu - v)|^2\, dx \le \int_{B_{1/\mu}(x_0)} (|\nabla v| + 5 \mu\, |v|)^2\, dx,
\]
and \eqref{56} follows from this since the restrictions of functions in $N_{l-1}$ to $\closure{B_{1/\mu_0}(x_0)}$ form a finite dimensional subspace of $C^1(\closure{B_{1/\mu_0}(x_0)})$ by the interior regularity of eigenfunctions. Proofs of \eqref{57}--\eqref{59} are similar.
\end{proof}

\begin{lemma} \label{Lemma 8}
We have
\begin{equation} \label{63}
\sup_{v \in N_{l-1} \cap S,\, s, t \ge 0}\, E(tv_\mu + su_{\eps,\,\mu}) < \frac{1}{4}\, S_4^2
\end{equation}
for all sufficiently small $\eps > 0$ and sufficiently large $\mu \ge \mu_0$.
\end{lemma}

\begin{proof}
For $v \in N_{l-1} \cap S$ and $s, t \ge 0$,
\begin{equation} \label{61}
E(tv_\mu + su_{\eps,\,\mu}) = E(tv_\mu) + E(su_{\eps,\,\mu})
\end{equation}
since $v_\mu$ and $u_{\eps,\,\mu}$ have disjoint supports. Since $a \le b$ and $(a,b) \in Q_l$,
\[
I(v,a,b) = 1 - \int_\Omega \left[a\, (v^-)^2 + b\, (v^+)^2\right] dx \le 1 - a \int_\Omega v^2\, dx \le 1 - \frac{a}{\lambda_{l-1}} < 0.
\]
Moreover,
\[
\frac{1}{\lambda_{l-1}} \le \int_\Omega v^2\, dx \le \vol{\Omega}^{1/2} \left(\int_\Omega v^4\, dx\right)^{1/2}
\]
and hence
\[
\inf_{v \in N_{l-1} \cap S}\, \int_\Omega v^4\, dx > 0.
\]
So it follows from Lemma \ref{Lemma 7} that $E(tv_\mu) \le 0$ for all sufficiently large $\mu \ge \mu_0$. Then \eqref{61} together with \eqref{60}--\eqref{62} gives for all sufficiently small $\eps > 0$,
\[
E(tv_\mu + su_{\eps,\,\mu}) \le E(su_{\eps,\,\mu}) \le \left[\half \left(1 - c_{30}\, \eps^2 \abs{\log \eps}\right) s^2 - \frac{1}{4} \left(1 - c_{31}\, \eps^4\right) s^4\right] S_4^2
\]
for some constants $c_{30}, c_{31} > 0$ depending on $\mu$. Maximizing the right-hand side over all $s \ge 0$ then gives
\[
E(tv_\mu + su_{\eps,\,\mu}) \le \frac{1}{4}\, \frac{\left(1 - c_{30}\, \eps^2 \abs{\log \eps}\right)^2}{1 - c_{31}\, \eps^4}\, S_4^2,
\]
from which the desired conclusion follows.
\end{proof}

We are now ready to prove Theorem \ref{Theorem 1} for $N = 4$.

\begin{proof}[Proof of Theorem \ref{Theorem 1} for\! $N = 4$]
As in the proof of Theorem \ref{Theorem 7}, if $\delta \in (0,1 - \max \set{a,b}\!/\!\lambda_{l+1})$ and $\rho > 0$ are sufficiently small,
\begin{equation} \label{10031}
\inf_{u \in A}\, E(u) > 0,
\end{equation}
where $A = \set{\theta(w,a/(1 - \delta),b/(1 - \delta)) + w : w \in M_{l-1} \cap S_\rho}$. We apply Theorem \ref{Theorem 3} taking $M = M_{l-1}$, $N = N_{l-1}$, $\theta = \theta(\cdot,a/(1 - \delta),b/(1 - \delta))$, $T : N_{l-1} \to H^1_0(\Omega)$ to be the bounded linear map given by
\[
Tv = v_\mu, \quad v \in N_{l-1},
\]
and $e = u_{\eps,\,\mu}$, with $\eps > 0$ sufficiently small and $\mu \ge \mu_0$ sufficiently large. Since $u_{\eps,\,\mu}$ and functions in $T(N_{l-1})$ have disjoint supports, $u_{\eps,\,\mu} \in H^1_0(\Omega) \setminus T(N_{l-1})$. We have
\[
\norm{I - T} = \sup_{v \in N_{l-1} \cap S} \left(\int_\Omega |\nabla (v_\mu - v)|^2\, dx\right)^{1/2} \to 0 \text{ as } \mu \to \infty
\]
by \eqref{56}, where $I$ is the identity map on $N_{l-1}$. If $\mu \ge \mu_0$ is sufficiently large, $E(tv_\mu) \le 0$ for all $v \in N_{l-1} \cap S$ and $t \ge 0$ as in the proof of Lemma \ref{Lemma 8}, so
\[
\sup_{u \in T(N_{l-1})}\, E(u) = 0.
\]
Together with \eqref{10031}, this gives the first inequality in \eqref{7}. The second inequality also holds by Lemma \ref{Lemma 8}, so \eqref{8} together with \eqref{10031} and \eqref{63} gives
\[
0 < c := \inf_{h \in \widetilde{H}}\, \sup_{u \in h(Q)}\, E(u) < \frac{1}{4}\, S_4^2,
\]
where $\widetilde{H} = \{h \in H : \restr{h}{T(N_{l-1})} = \id\}$. If $E$ satisfies the \PS{c} condition, then Theorem \ref{Theorem 3} gives a critical point of $E$ at the level $c$, which is nontrivial since $c > 0$. If, on the other hand, $E$ does not satisfy the \PS{c} condition, then $E$ has a \PS{c} sequence without a convergent subsequence, which then has a subsequence that converges weakly to a nontrivial critical point of $E$ (see Gazzola and Ruf \cite[Lemma 1]{MR1441856}).
\end{proof}

\section{Proofs of Theorems \ref{Theorem 9} and \ref{Theorem 10}} \label{T9&10}

In this section we prove Theorems \ref{Theorem 9} and \ref{Theorem 10} using Theorems \ref{Theorem 7} and \ref{Theorem 8}, respectively. As in the last section, we assume without loss of generality that $a \le b$. Problem \eqref{64} fits into the abstract setting of Sections \ref{DF spectrum} and \ref{Abs} as before. The potential
\[
F(u) = \half \int_\Omega \big(e^{u^2} - 1 - u^2\big)\, dx, \quad u \in H^1_0(\Omega)
\]
clearly satisfies $(F_1)$ and $(F_2)$. Moreover, the associated variational functional
\[
E(u) = \half \int_\Omega |\nabla u|^2\, dx - \half \int_\Omega \left[a\, (u^-)^2 + b\, (u^+)^2\right] dx - \half \int_\Omega \big(e^{u^2} - 1 - u^2\big)\, dx, \quad u \in H^1_0(\Omega)
\]
satisfies $(F_3)$ with
\[
c^\ast = 2 \pi
\]
(see de Figueiredo et al. \cite[Proposition 2.1]{MR1386960}).

Fix $x_0 \in \Omega$ and $0 < d_0 < \dist{x_0}{\bdry{\Omega}}$. Let
\[
\omega_{j,d}(x) = \frac{1}{\sqrt{2 \pi}} \begin{cases}
\sqrt{\log j} & \text{if } |x - x_0| \le d/j\\[7.5pt]
\dfrac{\log\, (d/|x - x_0|)}{\sqrt{\log j}} & \text{if } d/j < |x - x_0| < d, \qquad j \ge 2,\, 0 < d \le d_0\\[7.5pt]
0 & \text{if } |x - x_0| \ge d
\end{cases}
\]
be the Moser sequence concentrating at $x_0$. Note that $\omega_{j,d} \in H^1_0(\Omega)$ with $\norm{\omega_{j,d}} = 1$. We will apply Theorems \ref{Theorem 7} and \ref{Theorem 8} taking $e = \omega_{j,d}$ with $j \ge 2$ sufficiently large and $0 < d \le d_0$ sufficiently small. We have the estimates
\begin{gather}
\label{10037} \int_\Omega |\nabla \omega_{j,d}|\, dx \le c_{32}\, \frac{d}{\sqrt{\log j}},\\[15pt]
\label{10038} \int_\Omega \omega_{j,d}\, dx \le c_{33}\, \frac{d^2}{\sqrt{\log j}},\\[15pt]
\label{10035} \int_\Omega \omega_{j,d}^2\, dx \ge c_{34}\, \frac{d^2}{\log j}
\end{gather}
for some constants $c_{32}, \dots, c_{34} > 0$. Let $S = \set{u \in H^1_0(\Omega) : \norm{u} = 1}$.

\begin{lemma} \label{Lemma 50}
Let $K$ be a subset of $S \cap C^2(\Omega)$ such that
\begin{equation} \label{10044}
\sup_{u \in K}\, \norm[C^2(\closure{B_{d_0}(x_0)})]{u} < \infty.
\end{equation}
Then there exist constants $c_{35}, \dots, c_{37} > 0$ such that for all $j \ge 2$, $0 < d \le d_0$, $u \in K$, and $s, t \ge 0$,
\begin{equation} \label{10039}
\int_\Omega |\nabla (tu + s\, \omega_{j,d})|^2\, dx \le \left(1 + c_{35}\, \frac{d^2}{\sqrt{\log j}}\right)\! \left(t^2 + s^2\right),
\end{equation}
\begin{multline} \label{10041}
\int_\Omega \left[a \left((tu + s\, \omega_{j,d})^-\right)^2 + b \left((tu + s\, \omega_{j,d})^+\right)^2\right] dx \ge \bigg(\!\int_\Omega \left[a\, (u^-)^2 + b\, (u^+)^2\right] dx - c_{36}\, d^2\!\bigg)\, t^2\\[7.5pt]
+ c_{37}\, \frac{d^2}{\log j}\, s^2.
\end{multline}
In particular,
\[
I(tu + s\, \omega_{j,d},a,b) \le \Big(I(u,a,b) + c_{38}\, d^2\Big)\, t^2 + \left(1 + c_{35}\, \frac{d^2}{\sqrt{\log j}}\right) s^2
\]
for some constant $c_{38} > 0$.
\end{lemma}

\begin{proof}
We have
\[
\int_\Omega |\nabla (tu + s\, \omega_{j,d})|^2\, dx = t^2 + 2ts \int_\Omega \nabla u \cdot \nabla \omega_{j,d}\, dx + s^2
\]
since $u, \omega_{j,d} \in S$. Since $\omega_{j,d} = 0$ on $\bdry{\Omega}$,
\[
\int_\Omega \nabla u \cdot \nabla \omega_{j,d}\, dx = - \int_\Omega \omega_{j,d}\, \Delta u\, dx \le c_{35}\, \frac{d^2}{\sqrt{\log j}}
\]
for some constant $c_{35} > 0$ by \eqref{10044} and \eqref{10038}, so \eqref{10039} follows. Since $\omega_{j,d} = 0$ outside $B_d(x_0)$,
\begin{multline*}
\int_\Omega \left[a \left((tu + s\, \omega_{j,d})^-\right)^2 + b \left((tu + s\, \omega_{j,d})^+\right)^2\right] dx = t^2 \int_{\Omega \setminus B_d(x_0)} \left[a\, (u^-)^2 + b\, (u^+)^2\right] dx\\[7.5pt]
+ \int_{B_d(x_0)} \left[a \left((tu + s\, \omega_{j,d})^-\right)^2 + b \left((tu + s\, \omega_{j,d})^+\right)^2\right] dx = t^2 \int_\Omega \left[a\, (u^-)^2 + b\, (u^+)^2\right] dx\\[7.5pt]
+ \int_{B_d(x_0)} \left[a \left((tu + s\, \omega_{j,d})^-\right)^2 + b \left((tu + s\, \omega_{j,d})^+\right)^2 - a\, (tu^-)^2 - b\, (tu^+)^2\right] dx.
\end{multline*}
Since $a \le b$, the last integral is greater than or equal to
\begin{multline*}
\int_{B_d(x_0)} \left[a\, (tu + s\, \omega_{j,d})^2 - b\, (tu)^2\right] dx = - (b - a)\, t^2 \int_{B_d(x_0)} u^2\, dx + 2ats \int_{B_d(x_0)} u\, \omega_{j,d}\, dx\\[7.5pt]
+ as^2 \int_{B_d(x_0)} \omega_{j,d}^2\, dx \ge - c_{36}\, d^2\, t^2 + c_{37}\, \frac{d^2}{\log j}\, s^2
\end{multline*}
for some constants $c_{36}, c_{37} > 0$ by \eqref{10044}, \eqref{10038}, and \eqref{10035}, so \eqref{10041} also follows.
\end{proof}

Now we take $d = (\log j)^{-1/4}$. Then it follows from Lemma \ref{Lemma 50} that for all large $j \ge 2$, $u \in K$, and $s, t \ge 0$,
\begin{equation} \label{10055}
I(tu + s\, \omega_{j,(\log j)^{-1/4}},a,b) \le \left(I(u,a,b) + \frac{c_{38}}{\sqrt{\log j}}\right) t^2 + \left(1 + \frac{c_{35}}{\log j}\right) s^2.
\end{equation}
We have the following lemmas.

\begin{lemma} \label{Lemma 70}
Let $K$ be a subset of $S \cap C^2(\Omega)$ such that \eqref{10044} holds and
\begin{equation} \label{10049}
I(u,a,b) \le 0 \quad \forall u \in K.
\end{equation}
Then
\[
E(tu) \le 0 \quad \forall u \in K,\, t \ge 0,
\]
and for all sufficiently large $j \ge 2$,
\begin{equation} \label{10054}
\sup_{u \in K,\, s, t \ge 0}\, E(tu + s\, \omega_{j,(\log j)^{-1/4}}) < \infty
\end{equation}
and
\begin{equation} \label{10053}
E(tu + s\, \omega_{j,(\log j)^{-1/4}}) \to - \infty \text{ as } s^2 + t^2 \to \infty, \text{ uniformly in } u \in K.
\end{equation}
\end{lemma}

\begin{proof}
For $u \in K$ and $t \ge 0$,
\[
E(tu) = \frac{t^2}{2}\, I(u,a,b) - F(tu) \le 0
\]
by \eqref{10049} and $(F_2)$.

Set $v = tu + s\, \omega_{j,(\log j)^{-1/4}}$. Since $a, b > 0$ and $e^{v^2} \ge 1 + v^2 + v^4/2$,
\begin{equation} \label{10050}
E(v) \le \half \int_\Omega |\nabla v|^2\, dx - \frac{1}{4} \int_\Omega v^4\, dx \le \half \int_\Omega |\nabla v|^2\, dx - \frac{1}{4 \vol{\Omega}} \left(\int_\Omega v^2\, dx\right)^2
\end{equation}
by the H\"{o}lder inequality. Taking $a = b$ in \eqref{10041} gives
\begin{equation} \label{10051}
\int_\Omega v^2\, dx \ge \left(\int_\Omega u^2\, dx - \frac{c_{39}}{\sqrt{\log j}}\right) t^2 + \frac{c_{40}}{(\log j)^{3/2}}\, s^2
\end{equation}
for some constants $c_{39}, c_{40} > 0$. For $u \in K$, \eqref{10049} together with the assumption that $a \le b$ gives
\[
1 \le \int_\Omega \left[a\, (u^-)^2 + b\, (u^+)^2\right] dx \le b \int_\Omega u^2\, dx
\]
and hence
\begin{equation} \label{10052}
\int_\Omega u^2\, dx \ge \frac{1}{b}.
\end{equation}
Combining \eqref{10050}, \eqref{10039}, \eqref{10051}, and \eqref{10052} gives
\[
E(v) \le \half \left(1 + \frac{c_{35}}{\log j}\right)\! \left(t^2 + s^2\right) - \frac{1}{4 \vol{\Omega}} \left[\left(\frac{1}{b} - \frac{c_{39}}{\sqrt{\log j}}\right) t^2 + \frac{c_{40}}{(\log j)^{3/2}}\, s^2\right]^2,
\]
from which \eqref{10054} and \eqref{10053} follow for all sufficiently large $j \ge 2$.
\end{proof}

\begin{lemma} \label{Lemma 60}
Let $K$ be a subset of $S \cap C^2(\Omega)$ such that \eqref{10044} holds and
\begin{equation} \label{10048}
\sup_{u \in K}\, I(u,a,b) < 0.
\end{equation}
Then there exists $j_0 \ge 2$ such that
\[
\sup_{u \in K,\, s, t \ge 0}\, E(tu + s\, \omega_{j_0,(\log j_0)^{-1/4}}) < 2 \pi.
\]
\end{lemma}

\begin{proof}
If the conclusion is false, then it follows from Lemma \ref{Lemma 70} that for all $j \ge 2$, there exist $u_j \in K$, $s_j > 0$, and $t_j \ge 0$ such that
\[
E(t_j u_j + s_j\, \omega_{j,(\log j)^{-1/4}}) = \sup_{u \in K,\, s, t \ge 0}\, E(tu + s\, \omega_{j,(\log j)^{-1/4}}) \ge 2 \pi.
\]
Set $v_j = t_j u_j + s_j\, \omega_{j,(\log j)^{-1/4}}$. Then
\begin{equation} \label{4.6}
E(v_j) = \half \int_\Omega |\nabla v_j|^2\, dx - \half \int_\Omega \left[a\, (v_j^-)^2 + b\, (v_j^+)^2\right] dx - \half \int_\Omega \big(e^{v_j^2} - 1 - v_j^2\big)\, dx \ge 2 \pi.
\end{equation}
Moreover, $\tau v_j \in \set{tu + s\, \omega_{j,(\log j)^{-1/4}} : u \in K,\, s, t \ge 0}$ for all $\tau \ge 0$ and $E(\tau v_j)$ attains its maximum at $\tau = 1$, so
\begin{equation} \label{4.7}
\restr{\frac{\partial}{\partial \tau}\, E(\tau v_j)}{\tau = 1} = E'(v_j)\, v_j = \int_\Omega |\nabla v_j|^2\, dx - \int_\Omega \left[a\, (v_j^-)^2 + b\, (v_j^+)^2\right] dx - \int_\Omega \big(e^{v_j^2} - 1\big)\, v_j^2\, dx = 0.
\end{equation}

Since $e^{v_j^2} \ge 1 + v_j^2$, it follows from \eqref{4.6} and \eqref{10055} that
\[
4 \pi \le I(v_j,a,b) \le \left(I(u_j,a,b) + \frac{c_{38}}{\sqrt{\log j}}\right) t_j^2 + \left(1 + \frac{c_{35}}{\log j}\right) s_j^2.
\]
In view of \eqref{10048}, this implies that for all sufficiently large $j \ge 2$,
\begin{equation} \label{10060}
4 \pi + c_{41}\, t_j^2 \le \left(1 + \frac{c_{35}}{\log j}\right) s_j^2
\end{equation}
for some constant $c_{41} > 0$. So
\begin{equation} \label{4.9.3}
s_j^2 - 4 \pi \ge - \frac{c_{42}}{\log j}
\end{equation}
and
\begin{equation} \label{10056}
t_j \le c_{43}\, s_j
\end{equation}
for some constants $c_{42}, c_{43} > 0$.

Since $a, b > 0$, \eqref{4.7} gives
\begin{equation} \label{4.10}
\int_\Omega v_j^2\, e^{v_j^2}\, dx \le \int_\Omega \left(|\nabla v_j|^2 + v_j^2\right) dx.
\end{equation}
For $|x - x_0| \le d/j$ and sufficiently large $j$,
\[
|v_j| \ge s_j\, \omega_{j,(\log j)^{-1/4}} - t_j\, |u_j| \ge \sqrt{\frac{\log j}{2 \pi}} \left(s_j - \frac{c_{44}\, t_j}{\sqrt{\log j}}\right) \ge c_{45}\, s_j\, \sqrt{\log j}
\]
for some constants $c_{44}, c_{45} > 0$ by \eqref{10044} and \eqref{10056}. So
\begin{multline}
\int_\Omega v_j^2\, e^{v_j^2}\, dx \ge c_{45}^2\, s_j^2\, \log j\, e^{(s_j - c_{44}\, t_j/\sqrt{\log j})^2\, \log j/2 \pi}\, \pi \left(\frac{d}{j}\right)^2\\[7.5pt]
= c_{46}\, s_j^2\, \sqrt{\log j}\, j^{[(s_j - c_{44}\, t_j/\sqrt{\log j})^2 - 4 \pi]/2 \pi}
\end{multline}
for some constant $c_{46} > 0$. On the other hand, by the Poincar\'e inequality, \eqref{10039}, and \eqref{10056},
\begin{equation} \label{10061}
\int_\Omega \left(|\nabla v_j|^2 + v_j^2\right) dx \le c_{47}\, s_j^2
\end{equation}
for some constant $c_{47} > 0$. Combining \eqref{4.10}--\eqref{10061} gives
\begin{equation} \label{10063}
j^{[(s_j - c_{44}\, t_j/\sqrt{\log j})^2 - 4 \pi]/2 \pi} \le \frac{c_{48}}{\sqrt{\log j}}
\end{equation}
for some constant $c_{48} > 0$.

For sufficiently large $j$, \eqref{10063} implies that
\[
s_j \le 2\, \sqrt{\pi} + \frac{c_{44}\, t_j}{\sqrt{\log j}},
\]
and combining this with \eqref{10060} gives
\[
t_j \le \frac{c_{49}}{\sqrt{\log j}}
\]
for some constant $c_{49} > 0$. So
\[
s_j\, t_j \le \frac{c_{50}}{\sqrt{\log j}}
\]
for some constant $c_{50} > 0$. This together with \eqref{4.9.3} gives
\[
j^{[(s_j - c_{44}\, t_j/\sqrt{\log j})^2 - 4 \pi]/2 \pi} \ge j^{[s_j^2 - 4 \pi - 2c_{44}\, s_j\, t_j/\sqrt{\log j}]/2 \pi} \ge j^{- c_{51}/\log j} = e^{- c_{51}}
\]
for some constant $c_{51} > 0$, contradicting \eqref{10063}.
\end{proof}

We are now ready to prove Theorems \ref{Theorem 9} and \ref{Theorem 10}.

\begin{proof}[Proof of Theorem \ref{Theorem 9}]
We apply Theorem \ref{Theorem 7} taking $e = \omega_{j,(\log j)^{-1/4}}$ with $j \ge 2$ sufficiently large. Functions $\omega_{j,(\log j)^{-1/4}}$ with different centers $x_0$ and sufficiently large $j$ have disjoint supports and are therefore linearly independent. Since $N_{l-1}$ is finite dimensional, it follows that $x_0 \in \Omega$ can be chosen so that $\omega_{j,(\log j)^{-1/4}} \notin N_{l-1}$.

To verify \eqref{20}, we use Lemma \ref{Lemma 60} with $K = S \cap N_{l-1}$. For $u \in K$,
\[
I(u,a,b) = 1 - \int_\Omega \left[a\, (u^-)^2 + b\, (u^+)^2\right] dx \le 1 - a \int_\Omega u^2\, dx \le 1 - \frac{a}{\lambda_{l-1}} < 0
\]
since $b \ge a > \lambda_{l-1}$, so \eqref{10048} holds. By the interior regularity of eigenfunctions, $N_{l-1} \subset C^2(\Omega)$. So the restrictions of functions in $N_{l-1}$ to $\closure{B_{d_0}(x_0)}$ form a finite dimensional subspace of $C^2(\closure{B_{d_0}(x_0)})$. Since the restrictions of functions in $K$ to $\closure{B_{d_0}(x_0)}$ is a subset of this subspace and is bounded in the $H^1$-norm, \eqref{10044} follows.
\end{proof}

\begin{proof}[Proof of Theorem \ref{Theorem 10}]
We apply Theorem \ref{Theorem 8} taking $e = \omega_{j,(\log j)^{-1/4}}$ with $j \ge 2$ sufficiently large. Gradients $I'(\omega_{j,(\log j)^{-1/4}},a,b)$ with different centers $x_0$ and sufficiently large $j$ have disjoint supports and are therefore linearly independent. By Proposition \ref{Proposition 3}, $I'(v + \tau(v,a,b),a,b) \perp M_l$ for all $v \in N_l$ and hence $\set{I'(v + \tau(v,a,b),a,b) : v \in N_l}$ is a subset of $N_l$. Since $N_l$ is finite dimensional, it follows that $x_0 \in \Omega$ can be chosen so that $\pm \omega_{j,(\log j)^{-1/4}} \notin B := \set{v + \tau(v,a,b) : v \in N_l}$.

To verify \eqref{21}, we use Lemma \ref{Lemma 60} with $K = S \cap B$. To see that \eqref{10044} holds, let $u \in K$. By Proposition \ref{Proposition 3}, $I'(u,a,b) = z_u$ for some $z_u \in N_l$. Then
\begin{equation} \label{10057}
\int_\Omega \nabla u \cdot \nabla \zeta\, dx - \int_\Omega \left(bu^+ - au^-\right) \zeta\, dx = \int_\Omega \nabla z_u \cdot \nabla \zeta\, dx \quad \forall \zeta \in H^1_0(\Omega),
\end{equation}
so $u$ is a weak solution of
\[
\left\{\begin{aligned}
- \Delta u & = bu^+ - au^- - \Delta z_u && \text{in } \Omega\\[10pt]
u & = 0 && \text{on } \bdry{\Omega}.
\end{aligned}\right.
\]
Testing \eqref{10057} with $z_u$ gives
\[
\norm{z_u}^2 \le \norm{u} \norm{z_u} + \left(a\, |u^-|_2 + b\, |u^+|_2\right) \pnorm[2]{z_u},
\]
and since $K \subset S$, this implies that $\set{z_u : u \in K}$ is bounded in $H^1_0(\Omega)$. By the interior regularity of eigenfunctions, $N_l \subset C^{2,\alpha}(\Omega)$. So the restrictions of functions in $N_l$ to $\closure{B_{d_0}(x_0)}$ form a finite dimensional subspace of $C^{2,\alpha}(\closure{B_{d_0}(x_0)})$. Since the restrictions of functions in the set $\set{z_u : u \in K}$ to $\closure{B_{d_0}(x_0)}$ is a subset of this subspace and is bounded in the $H^1$-norm, it follows that this set is also bounded in $C^{2,\alpha}(\closure{B_{d_0}(x_0)})$. Now it follows from standard arguments in elliptic regularity theory that $K$ is bounded in $C^{2,\alpha}(\closure{B_{d_0}(x_0)})$.

It remains to show that \eqref{10048} holds. Let $u = v + \tau(v,a,b) \in K$, where $v \in N_l$. Then
\[
1 = \norm{u}^2 = \norm{v}^2 + \norm{\tau(v,a,b)}^2 \le c_{52} \norm{v}^2
\]
for some constant $c_{52} > 0$ since $\tau$ is positive homogeneous (see \cite[Proposition 4.3.1]{MR3012848}), so
\begin{equation} \label{71}
\norm{v}^2 \ge \frac{1}{c_{52}}.
\end{equation}
Since $b > \mu_l(a)$ and $\mu_l$ is continuous,
\[
b/(1 + \delta) \ge \mu_l(a/(1 + \delta))
\]
if $\delta \in (0,\min \set{a,b}/\lambda_{l-1} - 1)$ is sufficiently small. Then
\[
m_l(a/(1 + \delta),b/(1 + \delta)) \le 0
\]
(see \eqref{30}) and hence
\[
I(z,a/(1 + \delta),b/(1 + \delta)) \le 0,
\]
where $z = v + \tau(v,a/(1 + \delta),b/(1 + \delta))$ (see \eqref{28}). By Proposition \ref{Proposition 3}, then
\begin{equation} \label{72}
I(u,a,b) \le I(z,a,b) = (1 + \delta)\, I(z,a/(1 + \delta),b/(1 + \delta)) - \delta \norm{z}^2 \le - \delta \norm{z}^2.
\end{equation}
Since
\[
\norm{z}^2 = \norm{v}^2 + \norm{\tau(v,a/(1 + \delta),b/(1 + \delta))}^2 \ge \norm{v}^2,
\]
it follows from \eqref{72} and \eqref{71} that
\[
I(u,a,b) \le - \frac{\delta}{c_{52}},
\]
so \eqref{10048} holds.
\end{proof}

\bigskip

\noindent{\bf Acknowledgement}

\medskip

\noindent The second author was partially supported by MIUR--PRIN project ``Qualitative and quantitative aspects of nonlinear PDEs'' (2017JPCAPN\underline{\ }005).

{\small \def\cdprime{$''$}
}


\begin{thebibliography}{10}

\bibitem{MR709644}
Ha{\"{\i}}m Br{\'e}zis and Louis Nirenberg.
\newblock Positive solutions of nonlinear elliptic equations involving critical
  {S}obolev exponents.
\newblock {\em Comm. Pure Appl. Math.}, 36(4):437--477, 1983.

\bibitem{MR1011156}
Nguy{\^e}n~Phuong C{\'a}c.
\newblock On nontrivial solutions of a {D}irichlet problem whose jumping
  nonlinearity crosses a multiple eigenvalue.
\newblock {\em J. Differential Equations}, 80(2):379--404, 1989.

\bibitem{MR831041}
A.~Capozzi, D.~Fortunato, and G.~Palmieri.
\newblock An existence result for nonlinear elliptic problems involving
  critical {S}obolev exponent.
\newblock {\em Ann. Inst. H. Poincar\'e Anal. Non Lin\'eaire}, 2(6):463--470,
  1985.

\bibitem{MR1181350}
Mabel Cuesta and Jean-Pierre Gossez.
\newblock A variational approach to nonresonance with respect to the {F}u\v cik
  spectrum.
\newblock {\em Nonlinear Anal.}, 19(5):487--500, 1992.

\bibitem{MR499709}
E.~N. Dancer.
\newblock On the {D}irichlet problem for weakly non-linear elliptic partial
  differential equations.
\newblock {\em Proc. Roy. Soc. Edinburgh Sect. A}, 76(4):283--300, 1976/77.

\bibitem{MR628124}
E.~N. Dancer.
\newblock Corrigendum: ``{O}n the {D}irichlet problem for weakly nonlinear
  elliptic partial differential equations'' [{P}roc. {R}oy. {S}oc. {E}dinburgh
  {S}ect. {A} {\bf 76} (1976/77), no. 4, 283--300; {MR} {\bf 58} \#17506].
\newblock {\em Proc. Roy. Soc. Edinburgh Sect. A}, 89(1-2):15, 1981.

\bibitem{MR1269657}
D.~G. de~Figueiredo and J.-P. Gossez.
\newblock On the first curve of the {F}u\v cik spectrum of an elliptic
  operator.
\newblock {\em Differential Integral Equations}, 7(5-6):1285--1302, 1994.

\bibitem{MR1386960}
D.~G. de~Figueiredo, O.~H. Miyagaki, and B.~Ruf.
\newblock Elliptic equations in {${\bf R}^2$} with nonlinearities in the
  critical growth range.
\newblock {\em Calc. Var. Partial Differential Equations}, 3(2):139--153, 1995.

\bibitem{MR2514055}
Marco Degiovanni and Sergio Lancelotti.
\newblock Linking solutions for {$p$}-{L}aplace equations with nonlinearity at
  critical growth.
\newblock {\em J. Funct. Anal.}, 256(11):3643--3659, 2009.

\bibitem{MR0447688}
Svatopluk Fu\v{c}\'{\i}k.
\newblock Boundary value problems with jumping nonlinearities.
\newblock {\em \v{C}asopis P\v{e}st. Mat.}, 101(1):69--87, 1976.

\bibitem{MR658734}
Thierry Gallou{\"e}t and Otared Kavian.
\newblock R\'esultats d'existence et de non-existence pour certains probl\`emes
  demi-lin\'eaires \`a l'infini.
\newblock {\em Ann. Fac. Sci. Toulouse Math. (5)}, 3(3-4):201--246 (1982),
  1981.

\bibitem{MR1441856}
Filippo Gazzola and Bernhard Ruf.
\newblock Lower-order perturbations of critical growth nonlinearities in
  semilinear elliptic equations.
\newblock {\em Adv. Differential Equations}, 2(4):555--572, 1997.

\bibitem{MR871108}
A.~C. Lazer and P.~J. McKenna.
\newblock Critical point theory and boundary value problems with nonlinearities
  crossing multiple eigenvalues. {II}.
\newblock {\em Comm. Partial Differential Equations}, 11(15):1653--1676, 1986.

\bibitem{MR965532}
Alan Lazer.
\newblock Introduction to multiplicity theory for boundary value problems with
  asymmetric nonlinearities.
\newblock In {\em Partial differential equations (Rio de Janeiro, 1986)},
  volume 1324 of {\em Lecture Notes in Math.}, pages 137--165. Springer,
  Berlin, 1988.

\bibitem{MR1077275}
C.~A. Magalh{\~a}es.
\newblock Semilinear elliptic problem with crossing of multiple eigenvalues.
\newblock {\em Comm. Partial Differential Equations}, 15(9):1265--1292, 1990.

\bibitem{MR1484910}
Caryl~Ann Margulies and William Margulies.
\newblock An example of the {F}u\v cik spectrum.
\newblock {\em Nonlinear Anal.}, 29(12):1373--1378, 1997.

\bibitem{MR3012848}
Kanishka Perera and Martin Schechter.
\newblock {\em Topics in critical point theory}, volume 198 of {\em Cambridge
  Tracts in Mathematics}.
\newblock Cambridge University Press, Cambridge, 2013.

\bibitem{MR0488128}
Paul~H. Rabinowitz.
\newblock Some critical point theorems and applications to semilinear elliptic
  partial differential equations.
\newblock {\em Ann. Scuola Norm. Sup. Pisa Cl. Sci. (4)}, 5(1):215--223, 1978.

\bibitem{MR640779}
Bernhard Ruf.
\newblock On nonlinear elliptic problems with jumping nonlinearities.
\newblock {\em Ann. Mat. Pura Appl. (4)}, 128:133--151, 1981.

\bibitem{MR1322614}
Martin Schechter.
\newblock The {F}u\v c\'\i k spectrum.
\newblock {\em Indiana Univ. Math. J.}, 43(4):1139--1157, 1994.

\bibitem{MR1636619}
Martin Schechter.
\newblock New linking theorems.
\newblock {\em Rend. Sem. Mat. Univ. Padova}, 99:255--269, 1998.

\bibitem{MR95k:58033}
Martin Schechter and Kyril Tintarev.
\newblock Pairs of critical points produced by linking subsets with
  applications to semilinear elliptic problems.
\newblock {\em Bull. Soc. Math. Belg. S\'er. B}, 44(3):249--261, 1992.

\bibitem{MR0463908}
Giorgio Talenti.
\newblock Best constant in {S}obolev inequality.
\newblock {\em Ann. Mat. Pura Appl. (4)}, 110:353--372, 1976.

\end{thebibliography}
\end{document}